\documentclass[11pt]{amsart}
\usepackage[utf8]{inputenc}
\usepackage{amsmath,amssymb,amsthm,amsfonts,amscd,mathtools}
\usepackage{tikz-cd}
\usepackage{pst-node}
\usepackage[margin=3cm]{geometry}
\usepackage{enumerate}
\usepackage{faktor}
\usepackage{hyperref}
\hypersetup{colorlinks=true}
\makeatletter
\newtheorem*{rep@theorem}{\rep@title}
\newcommand{\newreptheorem}[2]{%
\newenvironment{rep#1}[1]{%
 \def\rep@title{#2 \ref{##1}}%
 \begin{rep@theorem}}%
 {\end{rep@theorem}}}
\makeatother

\newtheorem{thm}{Theorem}[section]
\newtheorem{Lemma}[thm]{Lemma}
\newtheorem{Proposition}[thm]{Proposition}
\newtheorem{Corollary}[thm]{Corollary}

\newtheorem*{genas}{General Assumption}

\theoremstyle{definition}
\newtheorem{Definition}[thm]{Definition}

\newtheorem{Remark}[thm]{Remark}

\newtheorem{Example}[thm]{Example}
\newtheorem{Construction}[thm]{Construction}

\makeatletter

\title{Mustafin varieties, moduli spaces and tropical geometry}
\author{Marvin Anas Hahn}
\author{Binglin Li}

\newcommand{\Addresses}{{
  \bigskip
  \footnotesize

  Marvin Anas Hahn, \textsc{Department of Mathematics, University of T\"ubingen, 72076 Tübingen, Germany}\par\nopagebreak
  \textit{E-mail address}: \texttt{marvin-anas.hahn@uni-tuebingen.de}

  \medskip

  Binglin Li, \textsc{Department of Statistics, University of Georgia, GA 30602, USA}\par\nopagebreak
  \textit{E-mail address}: \texttt{binglinligeometry@uga.edu}

}}
\date{}

\begin{document}
\begin{abstract}
Mustafin varieties are flat degenerations of projective spaces, induced by a choice of an $n-$tuple of lattices in a vector space over a non-archimedean field. They were introduced by Mustafin \cite{mustafin1978nonarchimedean} in the 70s in order to generalise Mumford's groundbreaking work on the unformisation of curves to higher dimension. These varieties have a rich combinatorial structure as can be seen in pioneering work of Cartwright, Häbich, Sturmfels and Werner \cite{cartwright2011mustafin}. In this paper, we introduce a new approach to Mustafin varieties in terms of images of rational maps, which were studied in \cite{Li13}. Applying tropical intersection theory and tropical convex hull computations, we use this method to give a new combinatorial description of the irreducible components of the special fibers of Mustafin varieties. This enables connections to various topics. In particular, we see that any multiview variety in the sense of \cite{zbMATH06249021} appears as an irreducible component of the special fiber of some Mustafin variety. Furthermore, we use an interpretation of Mustafin varieties as a moduli functor introduced by Faltings\cite{faltings2001toroidal} to relate them to certain moduli functors, called linked Grassmannians \cite{osserman2006limit}. These objects are featured in limit linear series theory. The focal point of study regarding linked Grassmannians are so-called \textit{simple points}. As a direct consequence of the new combinatorial description of Mustafin varieties, we prove that the simple points of linked Grassmannians are dense in every fiber. Finally, we use the connection to linked Grassmannians, to relate the special fibers of Mustafin varieties to certain local models of unitary Shimura varieties appearing in \cite{gortz2001flatness}.
\end{abstract}
\maketitle
\tableofcontents

\section{Introduction}
Mustafin varieties are flat degenerations of projective spaces induced by choosing an $n$-tuple of lattices in the Bruhat-Tits building $\mathfrak{B}_d$ associated to $\mathrm{PGL}(V)$ over a non-archimedean field $K$. These objects were introduced by Mustafin in \cite{mustafin1978nonarchimedean} in order to generalise Mumford's groundbreaking work on uniformisation of curves to higher dimensions \cite{zbMATH03362020}. Since then they have been repeatedly studied under the name \textit{Deligne schemes} (see e.g. \cite{faltings2001toroidal}, \cite{keel2006geometry}, \cite{cartwright2010hilbert}). By studying degenerations of projective spaces, we give a framework for the study of degeneations of projective subvarieties. In his original work, Mustafin studied the case of so-called \textit{convex point configurations} in $\mathfrak{B}_d$ as defined in Definition \ref{def:conv}. An approach to study arbitrary point configurations was developed in \cite{cartwright2011mustafin}, where the total space of this type of degenerations was named \textit{Mustafin variety} for the first time. There it was proved that if the lattices in the point configuration have diagonal form with respect to a common basis (i.e. they lie in the same apartment), the corresponding Mustafin variety is essentially a toric degeneration given by mixed subdivisions of a scaled simplex. These mixed subdivisions are beautiful combinatorial objects that are known to be equivalent to tropical polytopes and triangulations of products of simplices. For point configurations that do not obey this property some first structural results were proved.
In this paper, we give a new combinatorial description of the special fibers of Mustafin varieties, which yields a complete classification of the irreducible 
components of special fibers of Mustafin varieties. This allows a connection to multiview geometry. Moreover, we uncover a link between Mustafin varieties and so-called (pre)linked Grassmannians --- objects that show up in the theory of limit linear series. Using this connection, we further relate Mustafin varieties to the standard local model of Shimura varieties (as studied e.g. in \cite{gortz2001flatness}).\vspace{\baselineskip}

Let $R$ be a discrete valuation ring, $K$ the quotient field and $k$ the residue field. We fix a uniformiser $\pi$. As an example take $\mathbb{K}=\mathbb{C}((\pi))$ as the ring of formal Laurent series over $\mathbb{C}$ with discrete valuation $v(\sum_{n\ge l}a_n\pi^n)=l$ for $l\in\mathbb{Z}$ and $a_n\in\mathbb{C}$ with $a_l\neq0$. Then $R=\{\sum_{n\ge l}a_n\pi^n:k\in\mathbb{Z}_{\ge0}\}$ and $k=\mathbb{C}$. Moreover, let $V$ be vector space of dimension $d$ over $K$. We define $\mathbb{P}(V)=\mathrm{Proj}\mathrm{Sym}(V^*)$ as paramatrising lines through $V$. We call free $R-$modules $L\subset V$ of rank $d$ \textit{lattices} and define $\mathbb{P}(L)=\mathrm{Proj}\mathrm{Sym}(L^*)$, where $L^*=\mathrm{Hom}_R(L,R)$. Note, that we will only consider lattices up to homothety, i.e. $L\backsim L'$ if $L=c\cdot L'$ for some $c\in K^{\times}$.
\begin{Definition}[Mustafin varieties]
\label{def:musta}
 Let $\Gamma=\{L_1,\dots,L_n\}$ be a set of rank $d$ lattices in $V$. Then $\mathbb{P}(L_1),\dots,\mathbb{P}(L_n)$ are projective spaces over $R$ whose generic fibers are canonically isomorphic to $\mathbb{P}(V)\simeq\mathbb{P}^{d-1}_{\mathbb{K}}$. The open immersions $\mathbb{P}(V)\hookrightarrow\mathbb{P}(L_i)$ give rise to a map \[\mathbb{P}(V)\longrightarrow\mathbb{P}(L_1)\times_R\dots\times_R\mathbb{P}(L_n).\] We denote the closure of the image endowed with the reduced scheme structure by $\mathcal{M}(\Gamma)$. We call $\mathcal{M}(\Gamma)$ the \textit{associated Mustafin variety}. Its special fiber $\mathcal{M}(\Gamma)_k$ is a scheme over $k$.
\end{Definition}
While the generic fiber of such a scheme is isomorphic to $\mathbb{P}^{d-1}$, the special fiber has many interesting properties.\\
The main tool in this paper is the study of closures images of rational maps of the form \[f:\mathbb{P}(W)\dashrightarrow\mathbb{P}\left(\faktor{W}{W_1}\right)\times\cdots\times\mathbb{P}\left(\faktor{W}{W_n}\right),\] where $W$ is a vector space over $k$ of dimension $d$ and $(W_i)_{i\in[n]}$ is a tuple of sub-vector spaces $W_i\subset W$, such that $\bigcap W_i=\langle0\rangle$\cite{Li13}. We denote the closure of the above map by $X(W,W_1,\dots,W_n)$. \vspace{\baselineskip}

We make the following assumption for the rest of the paper.

\begin{genas}
The residue field $k$ is algebraically closed.
\end{genas}

We proceed as follows: Let $\Gamma$ be a point configuration in the Bruhat-Tits building $\mathfrak{B}_d$ associated to $\mathrm{PGL}(V)$. We consider the convex hull $\mathrm{conv}(\Gamma)$ (see Definition \ref{def:conv}), which is a set of lattices. To each lattice class $[L]\in\mathrm{conv}(\Gamma)$, we associate a variety $X_{\Gamma,[L]}$ of the form $X(k^d,W_1,\dots,W_n)$ for some $W_i$ depending on $[L]$ and $\Gamma$ (see Construction \ref{con:var}). Then we define two varieties in $\left(\mathbb{P}_{k}^{d-1}\right)^n$ as follows \[\widetilde{\mathcal{M}}(\Gamma)=\bigcup_{[L]\in\mathrm{conv}(\Gamma)}X_{[L]}\textrm{ and }\widetilde{\mathcal{M}}^r(\Gamma)=\bigcup_{[L]\in V(\mathrm{conv}(\Gamma))}X_{[L]},\]
where $V(\Gamma)$ is the set of polyhedral vertices of $\mathrm{conv}(\Gamma)$.
The following is one of the main results of this paper.
\begin{thm}
\label{thm:equ}
The irreducible components of Mustafin varieties are related to images of rational maps as follows:
\begin{enumerate}
\item If $\Gamma$ is a an arbitrary point configuration, we have \[\mathcal{M}(\Gamma)_k=\widetilde{\mathcal{M}}(\Gamma).\]
\item If $\Gamma$ is a point configuration in one apartment, we have \[\mathcal{M}(\Gamma)_k=\widetilde{\mathcal{M}}^r(\Gamma).\]
\end{enumerate}
\end{thm}
In each case, it is easy to see that the right hand side is contained in the special fiber of the Mustafin variety. For the other direction, we have to identify those lattice points that actually contribute an irreducible component. This is done by means of tropical intersection theory and multidegrees. Using this description we get a complete classification of the irreducible components of special fibers of Mustafin varieties as each variety of the form $X(k^d,W_1,\dots,W_n)$ occurs as an irreducible component. Thus, we obtain the following application of Theorem \ref{thm:equ}:
\begin{thm}
\label{thm:class}
The varieties $X(k^d;W_1,\dots,W_d)$ classify all irreducible components of special fibers of Mustafin varieties, i.e.
\begin{enumerate}
\item Any irreducible component of the special fibers of a Mustafin varieties is a variety of the form $X(k^d;W_1,\dots,W_d)$, and
\item every variety $X(k^d;W_1,\dots,W_n)$ appears as an irreducible component of $\mathcal{M}(\Gamma)_k$ for some $\Gamma$.
\end{enumerate}
\end{thm}
As mentioned before, we show that Mustafin varieties are closely related to so-called (pre)linked Grassmannians. \textit{Linked Grassmannians} were introduced in \cite{osserman2006limit} in the context of limit linear series. Osserman introduced a new theory of limit linear series that in a certain sense compactifies the Eisenbud-Harris limit linear series theory. A generalisation of this notion was introduced in \cite{osserman2014limit}, as so-called prelinked Grassmannians. (Pre)linked Grassmannians are degenerations of Grassmannians induced by a graph with additional data at the vertices at edges. If this graph is just a path, we call these objects linked Grassmannians and they were proved to be flat with reduced fibers in \cite{helm2008flatness}. The focal objects in studying (pre)linked Grassmannians are so-called \textit{simple points}. For a convex point configuration $\Gamma$ we associate a linked Grassmannian as a scheme $LG(1,\Gamma)$ over $R$. In \cite{faltings2001toroidal}, Faltings introduced a moduli functor for Mustafin varieties. Using this result, we can interpret Mustafin varieties as the moduli space for the linked Grassmannian problem. Furthermore, this proves that the class of linked Grassmannians induced by the data $\Gamma$ is flat with reduced fibers. Moreover, using a connection between Mustafin varieties and the simple points of a linked Grassmannian, the following theorem is a direct application of Theorem \ref{thm:equ}.
\begin{thm}
\label{thm:link}
The locus of simple points of $\mathrm{LG}(1,\Gamma)$ is dense in every fiber over $R$.
\end{thm}
Finally, we use our moduli space interpretation of Mustafin varieties to give a connection between Mustafin varieties and so-called \textit{local models of Shimura varieties associated to an (EL)-datum} as studied in \cite{gortz2001flatness}. Shimura varieties can be thought of as modular curves in higher dimension and are of importance in Langlands program.\\
This paper is structured as follows: In section \ref{sec:pre} we give a quick review of the tools needed to prove our theorems. In particular, we will give a summary for the notion of tropical convexity and the relation between Bruhat-Tits Buildings and tropical convexity. We state that for a fixed point configuration $\Gamma$ in one apartment, there is a bijection between the vertices in the convex hull of $\Gamma$ and the lattice points in the tropical hull of certain points in the tropical torus associated to $\Gamma$. We summarize some of the structural results for Mustafin varieties and introduce the necessary basis for (pre)linked Grassmannians in \ref{sec:link}. We finish the Preliminaries with a brief introduction to local models of Shimura varieties. Section \ref{sec:const} consists of the proof of Theorem \ref{thm:equ} and Theorem \ref{thm:class}. In subsection \ref{subsec:const}, we construct the varieties $\widetilde{\mathcal{M}}(\Gamma)$ and $\widetilde{\mathcal{M}}^r(\Gamma)$. We prove Theorem \ref{thm:equ} for the special case of $|\Gamma|=2$ in subsection \ref{sub:twola}. Moreover, we prove Theorem \ref{thm:equ} (1) in subsection \ref{sub:proof2}, Theorem \ref{thm:equ} (2) in subsection \ref{sec:proof1} and Theorem \ref{thm:class} in subsection \ref{sub:class}. In section \ref{sec:mulink}, we connect Mustafin varieties to prelinked Grassmannians by means of a moduli functor introduced by Faltings and prove Theorem \ref{thm:falt}. Moreover, we give an explicit description of the locus of simple points of the linked Grassmannians in terms of the special fibers of Mustafin varieties and prove Theorem \ref{thm:link}, Finally in subsection \ref{sec:shi} we interpret Mustafin varieties as a global realisation of the standard local model for Shimura varieties associtaed to an (EL)-datum.

\textbf{Acknowledgements.} We thank Hannah Markwig for her guidance and proof-reading throughout the preparation of this paper. Furthermore, we would like to thank Sarah Brodsky, Michael Joswig, Brian Osserman, Johannes Rau, Frank-Olaf Schreyer, Kristin Shaw and Bernd Sturmfels for many interesting discussions, comments and suggestions concerning this topic. Finally, we thank Annette Werner for pointing out a gap in Theorem \ref{thm:equ} (2) in an earlier version of this paper. The first author acknowledges partial support by the DFG collaborative research center TRR 195, project A11 (INST 248/235-1).
\section{Preliminaries}
\label{sec:pre}
\subsection{Tropical Geometry}
In this subsection, we recall some basics of tropical geometry required for this paper. Our main combinatorial tool in this paper is the notion of tropical convexity. We restrict ourselves to basic notions and results and refer to \cite{MR3287221} Chapter 5.2 for a more detailed introduction. Our proof of Theorem \ref{thm:equ} involves the identification of certain lattice points in so-called \textit{tropical convex hulls}. We achieve this by means of tropical intersection theory of tropical linear spaces (i.e. tropical varieties of degree $1$). Tropical intersection theory is a well-developed theory, for more details see e.g. \cite{allermann2010first} or \cite{MR3287221}.
\subsubsection{Tropical Convexity}
\label{sec:trop}
In a sense tropical convexity is the notion of convexity over the tropical semiring $(\overline{\mathbb{R}},\oplus,\odot)$, where $\overline{\mathbb{R}}=\mathbb{R}\cup\{\infty\}$, $a\oplus b=\mathrm{min}(a,b)$ and $a\odot b=a+b$. We make this more precise in the following definition:
\begin{Definition}
Let $S$ be a subset of $\mathbb{R}^n$. We call $S$ \textit{tropically convex}, if for any choice $x,y\in S$ and $a,b\in\mathbb{R}$ we get $a\odot x\oplus b\odot y\in S$.\\
The tropical convex hull of a given subset $V$ of $\mathbb{R}^n$ is given as the intersection of all tropically convex sets in $\mathbb{R}^n$ containing $S$. We denote the tropical convex hull of $V$ by $\mathrm{tconv}(V)$.
\end{Definition}

This definition implies that every tropical convex set $S$ is closed under tropical scalar multiplication. Thus, if $x\in S$ then so is $x+\lambda\textbf{1}$, where $\lambda\in\mathbb{R}$ and $\textbf{1}=(1,\dots,1)$. Therefore, we will usually identify $S$ with its image in $(n-1)$-dimension tropical torus $\faktor{\mathbb{R}^n}{\mathbb{R}\textbf{1}}$.

We are interested in tropical convex hulls of a finite number of points. We begin by treating the case of two points.

\begin{Proposition}[\cite{develin2004tropical}]
\label{prop:tropline}
 The tropical convex hull of two points $x,y\in\faktor{\mathbb{R}^n}{\mathbb{R}\textbf{1}}$ is a concatenation of at most $n-1$ ordinary line-segments. The direction of each line segment is a zero-one-vector.
\end{Proposition}

The proof of this proposition is constructive and describes the points in the tropical convex hull explicitly. We will use this fact in section \ref{sub:twola} to prove our statements in the case of a $2-$point configuration. To give this explicit description for $x=(x_1,\dots,x_n)$ and $y=(y_1,\dots,y_n)$, we note that after relabelling and adding multiples of $\textbf{1}$, we may assume $0=y_1-x_1\le y_2-x_2\le\dots\le y_n-x_n$. Then the tropical convex hull consists of the concatenation of the lines connecting the following points:
\begin{align*}
 x=&(y_1-x_1)\odot x\oplus y=(y_1,y_1-x_1+x_2,y_1-x_1+x_3,\dots,y_1-x_1+x_n)\\
 &(y_2-x_2)\odot x\oplus y=(y_1,y_2,y_2-x_2+x_3,\dots,y_2-x_2+x_n)\\
 &\cdots\cdots\\
 &(y_{n-1}-x_{n-1})\odot x\oplus y=(y_1,y_2,\dots,y_{n-1},y_{n-1}-x_{n-1}+x_n)\\
 &(y_n-x_n)\odot x\oplus y=(y_1,\dots,y_n)
\end{align*}
Some of these points might coincide, however they are always consecutive points on the line segment. Next, we introduce a useful description of tropical convex hulls in terms of bounded cells of a tropical hyperplane arrangement: Fix a set $\Gamma=\{v_1,\dots,v_n\}$ of $\faktor{\mathbb{R}^n}{\mathbb{R}\textbf{1}}$, with $v_i=(v_{i1},\dots,v_{in})$. Consider the standard tropical hyperplane at $v_i$ in the max-plus algebra:
\begin{align*}
H_{v_i}=\{w\in\faktor{\mathbb{R}^n}{\mathbb{R}\textbf{1}}:\ &\textrm{the \textit{maximum} of }w_1-v_{i1},\dots,w_n-v_{in}\\
&\textrm{ is attained at least twice}\}.
\end{align*}
Taking the common refinement, we obtain a polyhedral complex structure to $\faktor{\mathbb{R}^n}{\mathbb{R}\textbf{1}}$, i.e. a subdivision into convex polyhedra. The union of the bounded cells of this complex coincides with the tropical convex hull $\mathrm{tconv}(\Gamma)$ (see e.g. chapter 5.2 in \cite{MR3287221}).

\begin{figure}
\begin{center}
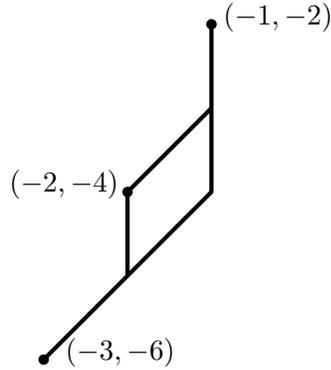
\caption{The tropical convex hull of $v_1=(0,-1,-2)$, $v_2=(0,-2,-4)$ and $v_3=(0,-3,-6)$ in the tropical torus.}
\label{fig:tconv}
\end{center}
\end{figure}

In the following example we compute the tropical convex hull of three points in the tropical torus.
\begin{Example}
\label{ex:tc}
 We pick 3 points 
\begin{align*}
 &v_1=(0,-1,-2),\\
 &v_2=(0,-2,-4),\\
 &v_3=(0,-3,-6).
 \end{align*}
 Viewed as points in the tropical torus, we can identify these points with $\widetilde{v}_1=(-1,-2)$, $\widetilde{v}_1=(-2,-4)$ and $\widetilde{v}_3=(-3,-6)$. The tropical convex hull is illustrated in Figure \ref{fig:tconv}.
\end{Example}

\begin{Remark}
The tropical convex hull of finitely many points is also called a \textit{tropical polytope}. Tropical polytopes can be thought of as tropicalisations of polytopes over the field of real Puiseux series $\mathbb{R}\{\{t\}\}$  (see Proposition 2.1 in \cite{zbMATH05246540}). One can generalise this notion to arbitrary tropical polyhedra and polyhedra over $\mathbb{R}\{\{t\}\}$, which in turn has applications in linear programming and complexity theory (see e.g. \cite{zbMATH06514534}).
\end{Remark}

\subsubsection{Stable Intersection}

Tropical intersection theory is a well-developed field. In this paper, we will only intersect standard tropical hyperplanes, which is why we restrict ourselves to this case. A more general discussion can be found in \cite{}.\par 

\begin{Definition}
We fix $n$ points $v_1,\dots,v_n\in\faktor{\mathbb{Z}^d}{\mathbb{Z}\textbf{1}}$ $H_i$ be the standard tropical hyperplane at $v_i$. We denote the \textit{set-theoretic intersection of $H_1,\dots,H_n$} by
\begin{equation*}
H_1\cap\dots\cap H_n.
\end{equation*}
Moreover, we define the \textit{stable intersection of $H_1,\dots,H_n$} by
\begin{align*}
H_1\cap_{st}\dots\cap_{st}H_n=\lim_{\epsilon_1,\dots,\epsilon_n\to 0}(H_1+\epsilon_1\cdot w_1)\cap\dots\cap(H_n+\epsilon_n\cdot w_n),
\end{align*}
where $+$ denotes the Minkowski sum and $w_i\in\faktor{\mathbb{Z}^d}{\mathbb{Z}\textbf{1}}$ are generic vertices.
\end{Definition}

\begin{Remark}
We note, that there are several definitions of stable intersection in tropical geometry, which all turn out to be equivalent. For various viewpoints, we refer to \cite{allermann2010first,rau2008intersections,katz2012tropical,jensen2016stable, shaw2013tropical}.
\end{Remark}

\begin{Example}
We illustrate the difference between set-theoretic intersection and stable intersection in the example of two lines not in tropical general position. The two lines in Figure \ref{fig:intersect} intersect set-theoretically in the half-bounded line segment as illustrated in the upper right. However, the stable intersection only yields a single point as illustrated in the lower right.\end{Example}

\begin{figure}
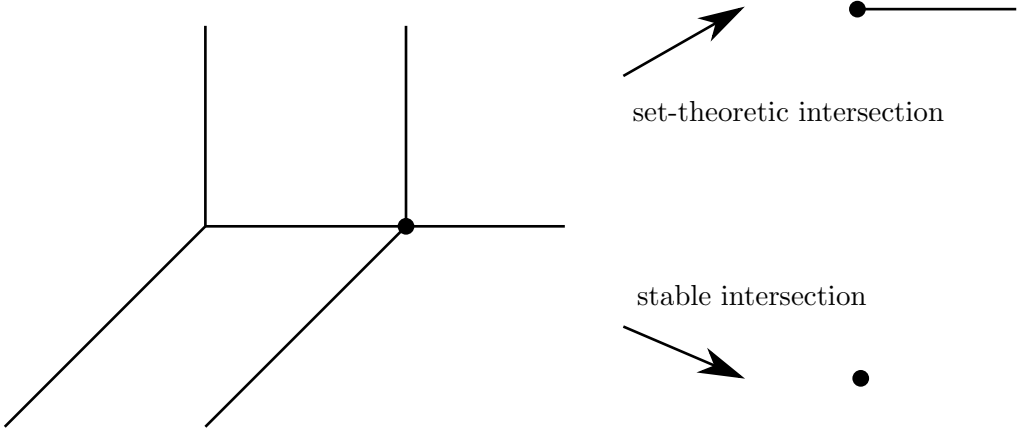
\caption{The difference between set-theoretic and stable intersection.}
\label{fig:intersect}
\end{figure}

In algebraic geometry, two general linear spaces of respective codimension $m_1$ and $m_2$ intersect in codimension $m_1+m_2$. A similar fact holds for tropical linear spaces. We first introduce a notion of points in \textit{tropical general position}.

\begin{Definition}
 A square matrix $M\in \mathbb{R}^{r\times r}$ is \textit{tropically singular}, if the minimum in \[\mathrm{det}(M)=\bigoplus_{\sigma\in\mathcal{S}_r}m_{1\sigma(1)}\odot\cdots\odot m_{r\sigma(r)}\] is attained at least twice. A point configuration $m_1,\dots,m_n$ in $\mathbb{R}^{d}$ is in \textit{tropical general position}, if every maximal minor of the matrix $((m_{ij})_{ij})$ is non-singular.
\end{Definition}


Let $H$ be the standard tropical hyperplane at a point $v$ in $\faktor{\mathbb{R}^n}{\mathbb{R}\textbf{1}}$, then we denote the $k-$fold stable self-intersection (i.e. $\underbrace{H\cap_{st}\dots\cap_{st}H}_\text{\textit{k}\;times}$) by $H^k$. Moreover, for a polyhedral complex $P$ of dimension $d$, we call its subcomplex $P'$ consisting of all polyhedra in $P$ of dimension smaller or equal to $k$, where $k<d$, \textit{the $k-$skeleton of $P$}. The following fact was proved in \cite{allermann2010first}.

\begin{Lemma}
Let $H$ be the standard tropical hyperplane in $\faktor{\mathbb{R}^n}{\mathbb{R}\textbf{1}}$ at $v$. Its $k-$fold stable self-intersection $H^k$ is given by its $(d-k)$-skeleton.
\end{Lemma}

\begin{Remark}
Let $v_1,\dots,v_n\in\faktor{\mathbb{Z}^d}{\mathbb{Z}\textbf{1}}$, let $m_1,\dots,m_n\in\mathbb{Z}_{\ge0}$ and let $H_i$ be the standard tropical hyperplane at $v_i$. Then the stable intersection coincides with set-theoretic intersection in the sense that \[H_1^{m_1}\cap_{st}\dots\cap_{st} H_n^{m_n}=H_1^{m_1}\cap\dots\cap H_n^{m_n}.\]
\end{Remark}


We end this section, with the following proposition.

\begin{Proposition}
\label{prop:tropint}
Let $v_1,\dots,v_n\in\faktor{\mathbb{Z}^d}{\mathbb{Z}\textbf{1}}$ and let $m_1,\dots,m_n\in\mathbb{Z}_{\ge0}$, such that $\sum_{i=1}^nm_i=d-1$. Further, let $H_i$ be standard tropical hyperplane at $v_i$, then \[H_1^{m_1}\cap_{st}\dots\cap_{st} H_n^{m_n}=\{pt\}\] is exactly one point with multiplicity one. If $v_1,\dots,v_n$ are in tropical general position, the set-theoric intersection \[H_1^{m_1}\cap\dots\cap H_n^{m_n}=\{pt\}\] coincides with the intersection product.
\end{Proposition}

\begin{proof}
This follows immediately from theorem 3.5 in \cite{jensen2016stable}, since $H_i^{m_i}$ is the tropicalisation of a linear space of co-dimension $m_i$.
\end{proof}

\subsection{Bruhat-Tits Buildings and Tropical Convexity}
\label{sec:bruh}
In this section, we recall some of the relations between Bruhat-Tits buildings and tropical convexity. For a short summary of Bruhat-Tits buildings, we refer to Section 2 of \cite{cartwright2011mustafin}, for more details on the relation between buildings and tropical convexity see e.g. \cite{zbMATH05900629}. We denote the Bruhat-Tits building associated to $\mathrm{PGL}(V)$ by $\mathfrak{B}_d$ and by $\mathfrak{B}_d^0$ the set of lattice classes in $\mathfrak{B}_d$. We call two lattice classes $[L],[M]$ in $\mathfrak{B}_d^0$ adjacent if there exist representative $L'\in[L]$ and $M'\in[M]$, such that $\pi M'\subset L'\subset M'$.\\
We pick a basis $e_1,\dots,e_d$ of $V$. The associated apartment $A$ is the geometric realisation of a simplicial complex on a vertex set isomorphic to the integral part $\faktor{\mathbb{Z}^d}{\mathbb{Z}\textbf{1}}$ of the tropical torus. The vertex set consists of all homothety classes of lattices in $V$ having diagonal form in $e_1,\dots,e_d$. More precisely, the following map is a bijection:
\begin{equation}
\label{map}
\begin{aligned}
 f: A\cap\mathfrak{B}_d^0&\longrightarrow\faktor{\mathbb{Z}^d}{\mathbb{Z}\textbf{1}}\\
 \{\pi^{m_1}Re_1+\dots+\pi^{m_d}Re_d\}&\longmapsto(-m_1,\dots,-m_d)+\mathbb{Z}\textbf{1}.
 \end{aligned}
\end{equation}
We write $\mathrm{tconv}(\Gamma)$ for the convex hull of the image of the point configuration under this map.
\begin{Definition}
\label{def:conv}
We call a point configuration $\Gamma\subset\mathfrak{B}_d^0$ convex if for $[L],[L']\in\Gamma$, any vertex of the form $[\pi^aL\cap\pi^bL']$ is also in $\Gamma$.\\
The convex hull $\mathrm{conv}(\Gamma)$ of a point configuration $\Gamma$ is the intersection of all convex sets containing $\Gamma$.
\end{Definition}

The following Lemma essentially goes back to \cite{keel2006geometry}, is a special case of Lemma 21 in \cite{zbMATH05225139} and can be found in this version as Lemma 4.1 in \cite{cartwright2011mustafin}.
\begin{Lemma}
\label{lem:brutrop}
 Let $\Gamma$ be a point configuration in one apartment. The map in (\ref{map}) induces a bijection between the lattices in $\mathrm{conv}(\Gamma)$ in the Bruhat-Tits building $\mathfrak{B}_d$ and the lattice points in $\mathrm{tconv}(\Gamma)$.
\end{Lemma}

\begin{Remark}
We call a point configuration $\Gamma\subset\mathfrak{B}_d^0$ that is contained in the same apartment $A$ in tropical general position, if the point configuration \[\{f(a):a\in A\}\subset\faktor{\mathbb{Z}^d}{\mathbb{Z}\textbf{1}}\subset\faktor{\mathbb{R}^d}{\mathbb{R}\textbf{1}}\] is in tropical general position.
\end{Remark}

\subsection{Images of rational maps}
\label{subsec:im}
Let $W$ be a $d-$dimensional vector space over $k$ and let $(W_i)_{i=1}^n$ be a tuple of subvectorspaces of $W$ with \[\bigcap_{i=1}^nW_i=\langle0\rangle.\] In \cite{Li13}, images of rational maps of the form \[\mathbb{P}(W)\dashrightarrow\mathbb{P}(\faktor{W}{W_1})\times\cdots\times\mathbb{P}(\faktor{W}{W_n})\] were studied. In particular, the Hilbert function and the multidegrees were computed. For every non-empty $I\subset[n]:=\{1,\dots,n\}$ we define \[d_I=\mathrm{dim}\bigcap_{i\in I}W_i.\]
For every $h\in\mathbb{Z}_{\ge0}$ we define $M(h)$ to be the set \[M(h)=\{(m_1,\dots,m_n)\in\mathbb{Z}_{\ge0}^n:\sum_{i=1}^nm_i=h\textrm{ and }d-\sum_{i\in I}m_i>d_I\textrm{ for every non-empty }I\subset[n]\}.\]

\begin{Remark}
There are several equivalent notions of the multidegree of a variety $X$ in a product of projective spaces $\left(\mathbb{P}_k^{d-1}\right)^n$. One possibility is to consider the multigraded Hilbert polynomial $h_X$: Let $x^u$ be a monomial of maximal degree in $h_X$ and $c_u$ be the coefficient. The multidegree function takes value $\frac{c_u}{u!}$ at $u$, where $u!=u_1!\cdots u_n!$.\vspace{\baselineskip}\\
Another way to describe the multidegree is to consider the intersection of $X$ with a system of $u_i$ general linear equations in the $i-$th factor of $\prod_{i=1}^n\mathbb{P}\left(\faktor{k^d}{W_i}\right)$, where we choose the $u_i$ such that the intersection is finite. Then the value of the multidegree function at $u=(u_1,\dots,u_n)$ is the cardinality of this intersection product. For a more thorough introduction in terms of Chow classes, see e.g. \cite{castillo2016representable}. We denote the set of multidegrees by \[\mathrm{multDeg}(X)=\{u\in\mathbb{Z}_{\ge0}^n:\textrm{the multidegree function is non-zero at }u\}.\]
\end{Remark}

\begin{thm}{\cite{Li13}}
\label{thm:rat}
Assume $k$ is algebraically closed. Set $p=\mathrm{max}\{h:M(h)\neq0\}$. The dimension of $X(W,W_1,\dots,W_n)$ is $p$. Its multidegree function takes value one at the integer vectors in $M(p)$ and $0$ otherwise. The Hilbert function of $X(W,W_1,\dots,W_n)$ is \[\sum_{S\subset M(p)}(-1)^{|S|-1}\prod_{j=1}^n{{u_i+\ell_{S,i}} \choose \ell_{S,i}},\] where the $u_i$ are the variables and $\ell_{S,i}$ is the smallest $i-$th components of all elements of $S$. Moreover, $X(W,W_1,\dots,W_n)$ is Cohen-Macaulay. 
\end{thm}

We end this section with an example.

\begin{Example}
 Let $V=\mathbb{C}^3$ and $e_1,e_2,e_3$ the standard basis. Moreover, let $V_1=(0),V_2=(e_2,e_3),V_3=(e_2,e_3)$ then \[X=X(V,V_1,V_2,V_3)=\mathbb{P}^2\times{pt}\times{pt}.\] In the notation of Theorem \ref{thm:rat}, we see $p=2=\mathrm{dim}(X)$ with multidegree $\mathrm{multDeg}(X)=M(p)=M(2)=\{(2,0,0)\}$. The multidegree function takes value $1$ at $(2,0,0)$ and $0$ else.
\end{Example}

\subsection{Mustafin Varieties}
\label{sec:must}
Recall Definition \ref{def:musta} for a Mustafin variety $\mathcal{M}(\Gamma)$ associate to a point configuration $\Gamma$ in $\mathfrak{B}_d^0$. In this subsection, we review the theory developed in \cite{cartwright2011mustafin}, where many interesting structural results about $\mathcal{M}(\Gamma)$ and its special fiber were proved. We state results needed for our approach and refer to \cite{cartwright2011mustafin} for a more detailed discussion.

\begin{Proposition}[\cite{cartwright2011mustafin},\cite{mustafin1978nonarchimedean}]
\label{prop:must}
\begin{enumerate}
 \item For a finite set of lattices $\Gamma$, the Mustafin variety $\mathcal{M}(\Gamma)$ is an integral, normal, Cohen-Macaulay scheme which is flat and projective over $R$. Its generic fiber is isomorphic to $\mathbb{P}^{d-1}_{\mathbb{K}}$ and its special fiber is reduced, Cohen-Macaulay and connected. All irreducible components are rational varieties and their number is at most ${n+d-2 \choose d-1}$, where $n=|\Gamma|$.
 \item If $\Gamma$ is a convex set in $\mathfrak{B}_d^0$, then the Mustafin variety is regular and its special fiber consists of $n$ smooth irreducible components that intersect transversely. In this case the reduction complex of $\mathcal{M}(\Gamma)$ is induced by the simplicial subcomplex of $\mathfrak{B}_d$ induced by $\Gamma$.
 \item If $\Gamma$ is a point configuration in one apartment, the components of  the special fiber correspond to maximal cells of a subdivision of the the simplex $|\Gamma|\cdot\Delta_{d-1}$ induced by $\Gamma$.
 \item An irreducible component mapping birationally to the special fiber of $\mathbb{P}(L_i)$ is called a \textit{primary component}. The other components are called \textit{secondary components}. For each $i=1,\dots,n$ there exists such a primary component. A projective variety $X$ arises as a primary component for some point configuration $\Gamma$ if and only if $X$ is the blow-up of $\mathbb{P}_k^{d-1}$ at a collection of $n-1$ linear subspaces.
 \item Let $C$ be a secondary component of $\mathcal{M}(\Gamma)_k$. There exists a vertex $v$ in $\mathfrak{B}_d^0$, such that \[\mathcal{M}(\Gamma\cup\{v\})_k\longrightarrow\mathcal{M}(\Gamma)_k\] restricts to a birational morphism $\tilde{C}\rightarrow C,$  where $\tilde{C}$ is the primary component of $\mathcal{M}(\Gamma\cup\{v\})_k$ corresponding to $v$.
\end{enumerate}
\end{Proposition}

We now choose coordinates in the same way as in \cite{cartwright2011mustafin}. Consider the diagonal map \[\Delta:\mathbb{P}(V)\longrightarrow\mathbb{P}(V)^n=\mathbb{P}(V)\times_K\cdots\times_K\mathbb{P}(V).\] The image of $\Delta$ is the subvariety of of $\mathbb{P}(V)^n$ cut out by the ideal generated by the $2\times2$ minors of a matrix $X=(x_{ij})_{\substack{i=1,\dots,d\\j=1,\dots,n}}$ of unknowns, where the $j$th column corresponds to coordinates in the $j$th factor.\\
Start with an element $g\in \mathrm{GL}(V)$, it is represented by an invertible $n\times n$ matrix over $K$. It induces a dual map $g^t:V^*\to V^*$ and thus a morphism $g:\mathbb{P}(V)\to\mathbb{P}(V)$. For $n$ elements $g_1,\dots,g_n\in\mathrm{GL}(V)$, the image of \[\mathbb{P}(V)\xrightarrow{\Delta}\mathbb{P}(V)^n\xrightarrow{g_1^{-1}\times\dots\times g_n^{-1}}\mathbb{P}(V)^n\] is the subvariety of $\mathbb{P}(V)^n$ cut out by the multihomogeneous prime ideal \[I_2((g_1,\dots,g_n)(X))\subset K[X],\] where $(g_1,\dots,g_n)(X)$ is the matrix whose $j$th column is given by \[g_j\begin{pmatrix}x_{1j}\\\vdots\\x_{dj}\end{pmatrix}.\]

Consider a reference lattice $L=Re_1+\cdots+Re_d$. For any configuration $\Gamma=\{L_1,\dots,L_n\}$ of lattices in $\mathfrak{B}_d^0$, we choose $g_i$, such that $g_iL=L_i$ for all $i$. The following diagram commutes:

\[\begin{tikzcd}
\mathbb{P}(V) \arrow[d] \arrow[rrr, "(g_1^{-1} \times\dots\times g_n^{-1})\circ\Delta"]& & & \mathbb{P}(V)^n \arrow[d]\\
\prod_{R}\mathbb{P}(L_i) \arrow[rrr, "(g_1^{-1} \times\dots\times g_n^{-1})"]& & & \mathbb{P}(L)^n
\end{tikzcd}\]

It follows immediately that the Mustafin Variety $\mathcal{M}(\Gamma)$ is isomorphic to the subscheme of $\mathbb{P}(L)^n\cong(\mathbb{P}^{d-1}_R)^n$ cut out by the multihomogeneous ideal $I_2\left((g_1,\dots,g_n)(X)\right)\cap R[X]$ in $R[X]$.

\begin{Definition}
A configuration $\Gamma=\{L_1,\dots,L_n\}$ in $\mathfrak{B}_d$ is of \textit{monomial type} if there exist bases for the R-modules $L_1,\dots,L_n$, such that the multi-homogeneous ideal in $k[X]$ defining $\mathcal{M}(\Gamma)_k$ is generated by monomials in the dual bases.
\end{Definition}

As proved in \cite{cartwright2011mustafin} point configurations $\Gamma$ in one apartment and tropical general position yield interesting properties of the special fiber of the corresponding Mustafin varieties. 

\begin{Proposition}[\cite{cartwright2011mustafin}]
Let $\Gamma$ be a point configurations. The following are equivalent:
 \begin{enumerate}
  \item The configuration $\Gamma$ is in general position.
 \item The special fiber $\mathcal{M}(\Gamma)_k$ is of monomial type.
  \item $\mathcal{M}_k$ is defined by a monomial ideal in $k[X]$ for the coordinates chosen before.
  \item The number of secondary components of $\mathcal{M}(\Gamma)_k$ equals ${{n+d-2}\choose{d-1}}-n$.
 \end{enumerate}
 Thus, in the case of tropical general position, the number of components of $\mathcal{M}(\Gamma)_k$ is ${{n+d-2}\choose{d-1}}$. 
\end{Proposition}

We finish this section by giving a moduli functor interpretation of Mustafin varieties.

We note that Faltings' notion of projective space is the dual construction to the one used here, i.e. in \cite{faltings2001toroidal} $\mathbb{P}(V)$ parametrises hyperplanes instead of lines. This yields the notion of quotient bundles in \textit{loc. cit.} instead of sub-bundles, which we use here:

\begin{thm}[\cite{faltings2001toroidal}]
\label{thm:falt}
Let $\Gamma$ be a convex point configuration. Then $\mathcal{M}(\Gamma)$ represents the following functor: To every $R-$scheme $S$ we associated the set of all tuples of line bundles $(l(L))_{L\in\Gamma}$, such that each inclusion $L_i\subset L_j$ (for $L_i,L_j\subset\Gamma$) maps $l(L_i)$ to $l(L_j)$.
\end{thm}

\subsection{(Pre)linked Grassmannians}
\label{sec:link}
In this subection, we discuss the basic theory surrounding (pre)linked Grassmannians, developed in the Appendix of \cite{osserman2006limit} and the Appendix of \cite{osserman2014limit} in the context of the theory of limit linear series. We begin with the following situation:\\
Let $S$ be any base scheme and $\mathcal{E}_1,\dots,\mathcal{E}_n$ vector bundles on $S$, each of rank $d$. We have morphisms $f_i:\mathcal{E}_i\rightarrow\mathcal{E}_{i+1}$, $g_i:\mathcal{E}_{i+1}\rightarrow\mathcal{E}_i$ and a positive integer $r<d$. For an $S-$scheme $T$, we denote the pull-backs of the vector bundles by $\mathcal{E}_{i,T}$ and the maps induced  by the pull-backs by $f_{i,T}:\mathcal{E}_{i,T}\to\mathcal{E}_{i+1,T}$. The functor studied in \cite{osserman2006limit} is described as follows:
\begin{Definition}
\label{def:fun}
We define $\mathcal{LG}(r,\{\mathcal{E}_i\}_i,\{f_i,g_i\}_i)$ to be the functor associating to each $S$-scheme $T$ the set of sub-bundles $V_1,\dots,V_n$ of $\mathcal{E}_{1,T},\dots,\mathcal{E}_{n,T}$ of rank $r$ that satisfy $f_{i,T}(V_i)\subset V_{i+1}$ and $g_{i,T}(V_{i+1})\subset V_{i}$ for all $i$.
\end{Definition}

\begin{figure}[h]
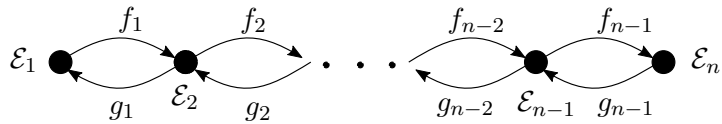
\caption{The underlying graph of the linked Grassmannian together with the input data.}
\end{figure}

This functor is representable by a projective scheme:
\begin{Lemma}
\label{lem:link}
The functor $\mathcal{LG}(r,\{\mathcal{E}_i\}_i,\{f_i,g_i\}_i)$ is representable by a projective scheme \[\mathrm{LG}=\mathrm{LG}(r,\{\mathcal{E}_i\}_i,\{f_i,g_i\}_i)\] over $S$, which is naturally a closed subscheme of a product $G$ of Grassmannians schemes over $S$; $G$ is smooth and projective over $S$ of relative dimension $nd(d-r)$. More precisely: Let $G_i$ be the Grassmannian of rank $r$ sub-bundles of $\mathcal{E}_i$, then $\mathrm{LG}$ is a closed subscheme of $G=G_1\times\cdots\times G_n$.
\end{Lemma}

\begin{Remark}
If there is no confusion about the data, we denote $\mathrm{LG}(r,\{\mathcal{E}_i\}_i,\{f_i,g_i\}_i)$ simply by $\mathrm{LG}$.
\end{Remark}

In order to study the dimension of $\mathrm{LG}$, additional hypothesis were made in \cite{osserman2006limit}.

\begin{Definition}
\label{def:link}
We say that $\mathrm{LG}(r,\{\mathcal{E}_i\}_i,\{f_i,g_i\}_i)$ is a \textbf{linked Grassmannian of length $n$} if $S$ is integral, Cohen-Macaulay and the following conditions on $f_i,g_i$ are satisfied:
\begin{enumerate}[(I)]
\item There exists an $s\in\Gamma(S,\mathcal{O}_S)$, such that for all $i$ \[f_i\circ g_i=s\cdot\mathrm{id}_{\mathcal{E}_{i+1}}\ \mathrm{and}\ g_i\circ f_i=s\cdot\mathrm{id}_{\mathcal{E}_i}.\]
\item On the fibers of the $\mathcal{E}_i$ at any point in the zero-locus of $s$, we have  \[\mathrm{Ker}f_i=\mathrm{Im}g_i\ \mathrm{and}\ \mathrm{Ker}g_i=\mathrm{Im}f_i.\]Equivalently, for any $i$ and given integers $r_1$ and $r_2$ such that $r_1+r_2<d$, the closed subscheme of $S$ obtained as the locus where $f_i$ has rank less than or equal to $r_1$ and $g_i$ has rank less than or equal to $r_2$ is empty.
\item At any point of $S$, we have \[\mathrm{Im}f_i\cap\mathrm{Ker}f_{i+1}=\{0\}\ \mathrm{and}\ \mathrm{Im}g_{i+1}\cap\mathrm{Ker}g_i=\{0\}.\]Equivalently, for any integer $r'$ and any $i$, we have locally closed subschemes of $S$ corresponding to the locus where $f_i$ has rank exactly $r'$ and $f_{i+1}f_i$ has rank less than or equal to $r'-1$ (similarly for $g_i$). Then we require all those subschemes to be empty.
\end{enumerate}
We will call a tuple $(\mathcal{E}_i,\{f_i,g_i\}_i)$ satisfying these conditions an $s-$linked chain. 
\end{Definition}

An important notion for our discussion are \textit{the exact points of a linked Grassmannian}.
\begin{Definition}
A point of a linked Grassmannian is exact if the corresponding  collection of subvectorspaces $V_i$ satisfy the conditions that $\left.\mathrm{ker}g_i\right|_{V_{i+1}}\subset f_i(V_i)$ and $\left.\mathrm{ker}f_i\right|_{V_{i}}\subset g_i(V_{i+1})$ for all $i$.
\end{Definition}

The following Proposition was proved in Lemma A.11 and Lemma A.12 in \cite{osserman2006limit}.
\begin{Proposition}
\label{prop:facts}
For linked Grassmannians, we have the following description of exact points:
\begin{enumerate}[(i)]
\item The exact points form an open subscheme of $\mathrm{LG}$ and are naturally described as the complement of the closed subscheme on which $\left.\mathrm{rk}f_i\right|_{V_i}+\left.\mathrm{rk}g_i\right|_{V_{i+1}}<r$ for some $i$.
\item In the case $s=0$, we can describe exact points as those with $\left.\mathrm{rk}f_i\right|_{V_i}+\left.\mathrm{rk}g_i\right|_{V_{i+1}}=r$ for all $i$ (which is even true for arbitrary scheme-valued points).
\end{enumerate}

Exact points have the following properties:
\begin{enumerate}[(i)]
\setcounter{enumi}{2}
\item The exact points are dense in $\mathrm{LG}$ and indeed dense in every fiber of $\mathrm{LG}\to S$.
\item Every exact point $x\in\mathrm{LG}$ is a smooth point of $\mathrm{LG}$ over $S$.
\end{enumerate}
The non-exact points of the fibers may also be described nicely:
\begin{enumerate}[(i)]
\setcounter{enumi}{4}
\item The non-exact points of a fiber  are precisely the intersections of the components of that fiber.
\end{enumerate}
\end{Proposition}

\begin{Remark}
In our case, linked Grassmannians will always be considered over $S=\mathrm{Spec}(R)$, where $R$ is a DVR. Thus, we will only talk about exact points and non-exact points of the special fiber, since the generic fiber is a Grassmannian as seen in the proof of Lemma \ref{lem:link} in \cite{osserman2006limit}.
\end{Remark}

The following theorem was proved in \cite{helm2008flatness}.

\begin{thm}[\cite{helm2008flatness}]
Let $S$ be integral and Cohen-Macaulay and let $\mathrm{LG}$ be any linked Grassmannian over $S$. Then $\mathrm{LG}$ is flat over $S$, reduced and Cohen-Macaulay, with reduced fibers.
\end{thm}

In the Appendix of \cite{osserman2014limit}, the notion of linked Grassmannians of length $n$ was generalized to arbitrary underlying graphs. We make this more precise in the following definition and give a quick summary of the results we will need in Section \ref{sec:mulink}.

\begin{Definition}
\label{def:prelink}
Let $G$ be a finite directed graph, connected by (directed) paths, $d$ an integer and $S$ a scheme. Suppose we are given the following data: $\mathcal{E}_{\bullet}$, consisting of vector bundles $\mathcal{E}_v$ of rank $d$ over $S$ for each $v\in V(G)$ and morphisms $f_{e}:\mathcal{E}_v\to\mathcal{E}_{v'}$ for each edge $e\in E(G)$, where $e$ points from $v$ to $v'$. For a directed path $P$ in $G$, we denote by $f_P$ the composition of the morphisms $f_e$ along th edges $e$ in the path.\\
Given an integer $r<d$ suppose further that we also have the following condition satisfied: For any two paths $P,P'$ in $G$ with the same head and tail, there are sections $s,s'\in\Gamma(S,\mathcal{O}_S)$ such that \[s\cdot f_P=s'\cdot f_{P'}.\] Moreover, if $P$ (resp. $P'$) is minimal (i.e. a shortest path with respect to the canonical graph metric), then $s'$ (resp. $s$) is invertible.\\
Then define the \textbf{prelinked Grassmannian} $\mathrm{LG}(r,\mathcal{E}_{\bullet})$ to be the scheme representing the functor associating to an $S-$scheme $T$ the set of all collections $(\mathcal{F}_v)_{v\in V(G)}$ of rank-$r$ subbundles of the $\mathcal{E}_{v,T}$ satisfying the property that for all edges $e\in E(G)$, we have $f_e(\mathcal{F}_v)\subset\mathcal{F}_{v'}$, where $e$ points from $v$ to $v'$. It was proved in \cite{osserman2014limit} that this functor is representable. Moreover, $\mathrm{LG}(r,\mathcal{E}_{\bullet})$ is a projective scheme over $S$ and compatible with base change.
\end{Definition}

\begin{Remark}
\begin{enumerate}
\item Note that we allow the path consisting of one vertex and consider $f_P=\mathrm{id}$ in this case. The condition implies that for any cycle $P$, we obtain $f_P=s\cdot \mathrm{id}$ for some scalar $s$, which corresponds to Definition \ref{def:link} (I).
\item When $G$ is the graph consisting of vertices $v_1,\dots,v_n$ and directed edges $e_{i,i+1}$ pointing from $v_i$ to $v_{i+1}$ and $e_{i+1,i}$ pointing from $v_{i+1}$ to $v_i$ we are in the situation of Definition \ref{def:fun}.
\end{enumerate}
\end{Remark}

The notion of exact points generalises easily to prelinked Grassmannians, but they are harder to study. To overcome this obstacle, the notion of simple points was introduced in \cite{osserman2014limit}. When the graph $G$ together with the data $\mathcal{E}_{\bullet},f_{\bullet}$ is an $s-$linked chain, the notion of simple points coincides with the notion of exact points (see \cite{osserman2014limit}).

\begin{Definition}
Let $k$ be a field over $S$ and $(F_v)_v$ a $k-valued$ point of a prelinked Grassmannian $\mathrm{LG}(r,\mathcal{E}_{\bullet})$. We say that $(F_v)_v$ is simple if there exist $v_1,\dots,v_r\in V(G)$ and $s_i\in F_{v_i}$ for $i=1,\dots,r$, such that for every $v\in V(G)$, there exist paths $P^{v_i}$ with each $P^{v_i}$ going from $v_i$ to $v$ and such that $f_{P^{v_1}}(s_1),\dots,f_{P^{v_r}}(s_r)$ form a basis for $F_v$.
\end{Definition}

We end this section by giving the two main statements for simple points:

\begin{Proposition}[\cite{osserman2014limit}]
The following statements hold for simple points of a prelinked Grassmannian:
\begin{enumerate}
\item The simple points form an open subset of $LG(r,\mathcal{E}_{\bullet})$.
\item On the locus of simple points in $LG(r,\mathcal{E}_{\bullet})$ is smooth over $S$ of relative dimension $r(d-r)$.
\end{enumerate}
\end{Proposition}

\subsection{Local models of Shimura varieties}
\label{subsec:shim}
In this subsection, we recall some of the basic theory around certain local models of Shimura varieties studied in \cite{gortz2001flatness}. Shimura varieties can be viewed as higher dimensional analogous of modular curves. For a detailed introduction see for example \cite{PRSLocal}. For the rest of this subsection, we assume that $k$ is a perfect field.
In \cite{gortz2001flatness} a local model for unitary Shimura varieties was studied. The standard local model is defined to be the $R-$scheme representing the following functor:\vspace{\baselineskip}

Fix a number $0<r<d$  and let $e_1,\dots,e_d$ be the canonical basis of $K^d$. We fix lattices $L_0,\dots,L_{d-1}$ as follows: $L_0=Re_1+\cdots+Re_d$ and $L_i=R\pi e_1+\cdots+R\pi e_i+R e_{i+1}+\cdots+R e_d$. We obtain the following chain complex by the natural inclusion maps
\[\begin{tikzcd}
\cdots \arrow[r] & L_0 \arrow[r] & L_1 \arrow[r] & \cdots \arrow[r] & L_{d-2} \arrow[r] & L_{d-1} \arrow[r] & \pi^{-1}L_0 \arrow[r] & \cdots
\end{tikzcd}\]
For each $R-$scheme $S$, the $S-$valued points of $M^{loc}$ are the isomorphism classes of the following commutative diagram:
\[\begin{tikzcd}
L_{0,S} \arrow[r] & L_{1,S} \arrow[r]& \cdots \arrow[r] & L_{d-1,S} \arrow[r,"\pi"] & L_{0,S}\\
\mathcal{F}_0 \arrow[r] \arrow[u, hook] & \mathcal{F}_1 \arrow[r] \arrow[u, hook] & \cdots \arrow[r] & \mathcal{F}_{n-1} \arrow[r] \arrow[u, hook] & \mathcal{F}_{n} \arrow[u, hook]
\end{tikzcd}\]
where $L_{i,S}=L_i\otimes S$ and the $\mathcal{F}$ are sub-vector bundles of rank $r$ of $L_{i,S}$.

\begin{thm}[\cite{gortz2001flatness}]
\label{thm:gortz}
The local model $M^{loc}$ is flat over $R$ and its special fiber is reduced.
\end{thm}

\section{Special fibers of Mustafin Varieties}
\label{sec:const}
The goal of this section is to prove Theorem \ref{thm:equ} and Theorem \ref{thm:class}. We begin by constructing the varieties $\widetilde{\mathcal{M}}(\Gamma)$ and $\widetilde{\mathcal{M}}^r(\Gamma)$.

\subsection{Constructing the varieties $\widetilde{\mathcal{M}}(\Gamma)$ and $\widetilde{\mathcal{M}}^r(\Gamma)$}
\label{subsec:const}
\begin{Construction}
 \label{con:var}
 Let $\Gamma=\{L_1,\dots,L_n\}$ be a point configuration $\mathfrak{B}_d^0$ as in section \ref{sec:bruh}.
Let $[L]$ be a homothety class in $\mathrm{conv}(\Gamma)$. We choose $L\in[L]$ to be the reference lattice in our choice of coordinates in Subsection \ref{sec:must}. We describe maps between $[L]$ and $[L_i]$: Pick an apartment $A_i$, corresponding to the basis $e_1^i,\dots,e_d^i$, such that $[L]$ and $[L_i]$ are contained in $A$, i.e.: \[L=\pi^{m_1}Re_1^i+\cdots+\pi^{m_d}Re_d^i\] and  \[L_i=\pi^{n^i_1}Re_1^i+\cdots+\pi^{n^i_d}Re_d^i.\] (This is always possible, see e.g. the discussion prior to Proposition 6.111 in \cite{buldingsbook}.) There is a canonical map $g_i$, such that $g_i[L]=[L_i]$, which up to homothety (i.e. multiplication by a scalar) is given by the matrix \[h_i=\begin{pmatrix}
\pi^{n^i_1-m_1}& &\\
& \ddots &\\
& & \pi^{n^i_d-m_d}
\end{pmatrix}\]
in the basis $e_1^i,\dots,e_d^i$. Thus, in our coordinates the Mustafin variety $\mathcal{M}(\Gamma)$ is given by the closure of the image of
\begin{align*}
f_{\Gamma,[L]}:\mathbb{P}(V)\xrightarrow{(g_1^{-1}\times\dots\times g_n^{-1})\circ\Delta}\mathbb{P}(L)^n.
\end{align*}
We choose a set of generators $e_1^{[L]},\dots,e_d^{[L]}$ for $L$ and we denote the transformation matrix representing $g_i$ with respect to $e_1^{[L]},\dots,e_d^{[L]}$ by $G_i'$. We define an invertible matrix $G_i$ by $G_i^{-1}=\pi^{s_i} G_i'$, where $s_i=-\min_{m,n=1,\dots,d}\{\mathrm{val}\left((G_i^{-1})_{m,n}\right)\}$.
This induces a rational map over the special fiber:
\begin{align*}
\widetilde{f_{\Gamma,[L]}}:\mathbb{P}_k^{d-1}\dashrightarrow(\mathbb{P}_k^{d-1})^n,
\end{align*}
given by $(\widetilde{g_1^{-1}}\times\dots\times\widetilde{g_1^{-1}})\circ\Delta$, where $\widetilde{g_i^{-1}}$ is obtained from $G_i^{-1}$ by the entry-wise residue map.
Denoting the varieties $X_{\Gamma,[L]}=\overline{\mathrm{Im}(\widetilde{f_{\Gamma,[L]}})}$, we define our desired varieties as the union of the closures of the images of these rational maps: \[\widetilde{\mathcal{M}}(\Gamma)=\bigcup_{[L]\in\mathrm{conv}(\Gamma)}X_{\Gamma,[L]}\textrm{ and }\widetilde{\mathcal{M}}^r(\Gamma)=\bigcup_{[L]\in V(\mathrm{conv}(\Gamma))}X_{\Gamma,[L]},\]
where $V(\mathrm{conv}(\Gamma)$ is the set of polyhedreal vertices of $\Gamma$.

\end{Construction}

\begin{Remark}
The expression $\widetilde{\mathcal{M}}(\Gamma)=\bigcup_{[L]\in\mathrm{conv}(\Gamma)}X_{\Gamma,[L]}\textrm{ and }\widetilde{\mathcal{M}}^r(\Gamma)=\bigcup_{[L]\in V(\mathrm{conv}(\Gamma))}X_{\Gamma,[L]}$ include a slight abuse of notation as $X_{\Gamma,[L]}$ is contained in the special fiber of $\mathbb{P}(L)^n$ and $X_{\Gamma,[L']}$ is contained in the special fiber of $\mathbb{P}(L')^n$. In order to take the union , we observe that the isomorphism $\mathbb{P}(L')^n\to\mathbb{P}(L)^n$ induces an isomorphism of special fibers and thus maps $X_{\Gamma,[L']}$ into the special fiber of $\mathbb{P}(L)^n$, where we can take the union.\vspace{\baselineskip}

When $\Gamma$ is a point configuration in one apartment induced by $e_1,\dots, e_d$, we can choose a set of generators $e_1^{[L]},\dots,e_d^{[L]}$, such that the matrices $g_i$ coincide with the matrices $h_i$ up to homothety. The maps $\widetilde{g_i^{-1}}$ are then given by \[\begin{pmatrix}
a_1 & &\\
& \ddots &\\
& & a_d 
\end{pmatrix},\]
where \[a_j=1, \textrm{ if }n_j^i-m_j=\min_{k=1,\dots,d}(n_k^i-m_k)\] and $a_j=0$ else. (see Example \ref{ex:must}).\vspace{\baselineskip}

The key observation for the proof of Theorem \ref{thm:equ} is that the maps $\widetilde{f_{\Gamma,[L]}}$ factorise as follows: \begin{equation}\label{equ:rat}\begin{tikzcd}
\mathbb{P}_k^{d-1}\arrow[rdd, dashrightarrow, "\widetilde{f_{\Gamma,[L]}}"] \arrow[r,dashrightarrow] & \mathbb{P}\left(\faktor{k^d}{\mathrm{Ker}(\widetilde{g_1^{-1}})}\right)\times\cdots\times\mathbb{P}\left(\faktor{k^d}{\mathrm{Ker}(\widetilde{g_n^{-1}})}\right) \arrow[dd,hookrightarrow]\\
\\
& \left(\mathbb{P}_k^{d-1}\right)^n
\end{tikzcd}\end{equation}
Thus, each variety $\overline{\mathrm{Im}(\widetilde{f_L})}$ is a variety of the form $X(k^d,W_1,\dots,W_n)$, where $W_i=\mathrm{ker}(\widetilde{g_i^{-1}})$. Note that $\bigcap_{i=1}^nW_i$ might not be trivial. However, the components of $\mathcal{M}(\Gamma)_k$ are equidimensional. The vertices with $\bigcap_{i=1}^nW_i\neq\langle0\rangle$ contribute varieties of lower dimension and thus are contained in an irreducible component by Lemma \ref{lem:con}. This irreducible component is contributed by a vertex satisfying $\bigcap_{i=1}^nW_i=\langle0\rangle$.
\end{Remark}
Before we prove Theorem \ref{thm:equ}, we illustrate our construction in the following example:
\begin{Example}
\label{ex:must}
By Lemma \ref{lem:brutrop}, there is a natural correspondence between lattice points in $\faktor{\mathbb{Z}^d}{\mathbb{Z}\textbf{1}}$ and lattices in $V$ over $R$. In Example 2.2 in \cite{cartwright2011mustafin}, the special fiber of the Mustafin variety corresponding to the vertices in Example \ref{ex:tc} was computed to be the union of the following irreducible components
 \[\mathbb{P}^2\times pt \times pt,\ pt \times\mathbb{P}^2\times pt,\ pt \times pt\times \mathbb{P}^2,\ \mathbb{P}^1\times\mathbb{P}^1\times pt,\ \mathbb{P}^1\times pt\times \mathbb{P}^1 ,\ pt\times\mathbb{P}^1\times\mathbb{P}^1\]                                                                                                                                                                                                                                                                                                                                                         
 Our construction yields the same variety as seen in the following computations.
 For $v=v_3$, we obtain the map \[\begin{pmatrix}1 & 0 & 0\\
                                           0 & \pi^2 & 0\\
                                           0 & 0 & \pi^4 \end{pmatrix}\times\begin{pmatrix}
                                           1 & 0 & 0\\
                                           0 & \pi & 0\\
                                           0 & 0 & \pi^2\end{pmatrix}\times \mathrm{id}\]                             
Modding out $\pi$, we immediately see that \[\mathrm{Im}(\widetilde{f_{\Gamma,[L]}})=pt\times pt \times \mathbb{P}^2,\]where $[L]$ is the lattice class corresponding to $v$ Analogously, we obtain $pt \times\mathbb{P}^2\times pt$ and $\mathbb{P}^2 \times pt\times pt$ for $v=v_2$ and $v=v_1$ respectively.\\
For $v=(0,-1,-4)$, we obtain the map \[\begin{pmatrix}1 & 0 & 0\\0 & 1 & 0\\0 & 0 & \pi^2 \end{pmatrix}\times\begin{pmatrix}\pi & 0 & 0\\0 & 1 & 0\\0 & 0 & \pi\end{pmatrix}\times\begin{pmatrix}\pi^2 & 0 & 0\\0 & 1 & 0\\0 & 0 & 1\end{pmatrix}.
\]
Modding out $t$, we immediately see that \[\mathrm{Im}(f_{\Gamma,[L]})=\mathbb{P}^1\times pt\times\mathbb{P}^1,\]where once again $[L]$ is the lattice class corresponding to $v$ Analogously, we obtain $\mathbb{P}^1\times \mathbb{P}^1\times pt$ and $pt\times\mathbb{P}^1\times\mathbb{P}^1$ for $v=(0,-1,-3)$ and $v=(0,-2,-4)$ respectively.
\end{Example}

The following Lemma implies the easier containment for the equalities of Theorem \ref{thm:equ}.

\begin{Lemma}
\label{lem:con}
 For any lattice $[L]$ in $\mathrm{conv}(\Gamma)$, the following relation holds:
 \[\overline{\mathrm{Im}(\widetilde{f_{\Gamma,[L]}})}\subset\mathcal{M}(\Gamma)_k.\]
 Moreover, $\overline{\mathrm{Im}(\widetilde{f_{\Gamma,[L]}})}$ is an irreducible component of $\mathcal{M}(\Gamma)_k$ if and only if\[\mathrm{dim}(\overline{\mathrm{Im}(\widetilde{f_{\Gamma,[L]}})})=\mathrm{dim}(\mathcal{M}(\Gamma)_k).\]
\end{Lemma}

\begin{proof}
To prove the first statement, denote the ideal in $R[x_{ij}]$ defining $\mathcal{M}(\Gamma)$ in $\mathbb{P}(L)^n$ by $I(\Gamma,[L])$ and the ideal in $k[x_{ij}]$ defining $\mathcal{M}(\Gamma)_k$ in the special fiber of $\mathbb{P}(L)^n$ by $I(\Gamma,[L])_k$. Let $v\in\mathbb{P}_k^{d-1}\hookrightarrow\mathbb{P}_K^{d-1}$, $w=\widetilde{f_{\Gamma,[L]}}(v)$ and let $g\in I(\Gamma,[L])_k$, we want to prove that $g(w)=0$. We choose a polynomial $g'\in I(\Gamma,[L])$, which specialises to $g$ and define $W'=f_{\Gamma,[L]}(v)\in\mathcal{M}(\Gamma)\subset\mathbb{P}(L)^n$. We see immediately that $g'(W')=0$. Moreover, since we chose the component-wise maps $G_i^{-1}$, such that $G_i^{-1}\in\mathrm{Mat}(R,d\times d)$ and such that there are elements in $G_i^{-1}$ of valuation $0$, we see that the constant term of $g'(W')$ is given by $g(W)$. However, since $g'(w')=0$, the constant term vanishes as well and thus $g(w)=0$. We obtain $\widetilde{f_{\Gamma,[L]}}(v)\in\mathcal{M}(\Gamma)_k$ and thus $\widetilde{f_{\Gamma,[L]}}\subset\mathcal{M}(\Gamma)_k$ as desired. By definition $\overline{\mathrm{Im}(\widetilde{f_{\Gamma,L}})}$ is reduced and irreducible and by \cite{cartwright2011mustafin} the same is true for the irreducible components of $\mathcal{M}(\Gamma)_k$. Therefore, the second statement follows.
\end{proof}

We further give a purely tropical description of the map $\widetilde{f_{\Gamma,[L]}}$, whenever $\Gamma$ is a point configuration in one apartment. In fact, we associate such a map and thus a variety to any point configuration in the tropical torus $\faktor{\mathbb{R}^n}{\mathbb{R}\textbf{1}}$.

\begin{Construction}
\label{con:general}
Let $\Gamma=\{[v_1],\dots,[v_n]\}\subset\faktor{\mathbb{R}^n}{\mathbb{R}\textbf{1}}$ be a point configuration in the tropical torus, where $v_i=(v_i^1,\dots,v_i^n)$. We define rational maps
\begin{equation}
f_{\Gamma,v_i}:\mathbb{P}_k^{d-1}\dashrightarrow(\mathbb{P}_k^{d-1})^n
\end{equation}
where the map to the $j$th factor is given by
\begin{equation}
\begin{pmatrix}
a_1 & &\\
& \ddots &\\
& & a_d 
\end{pmatrix},
\end{equation}
where $a_k=1$, if $v_j^k-v_i^k=\mathrm{min}_{l}(v_j^l-v_i^l)$ and $a_k=0$ otherwise. Then we define
\begin{equation}
\widetilde{M}(\Gamma)=\bigcup_{v\in\mathrm{conv}(\Gamma)}X_{\Gamma,v}\,\,\mathrm{and}\,\,\widetilde{M}^r(\Gamma)=\bigcup_{v\in V(\mathrm{conv}(\Gamma))}X_{\Gamma,v}.
\end{equation}
For $\Gamma'=\{L_1,\dots,L_n\}$ a point configuration of lattice classes in one apartment, let $\Gamma=\{[v_1],\dots,[v_n]\}$ be the corresponding point configuration in the tropical torus, then
\begin{equation}
\widetilde{M}(\Gamma)=\widetilde{M}(\Gamma')\,\,\mathrm{and}\,\,\widetilde{M}^r(\Gamma)=\widetilde{M}^r(\Gamma'),
\end{equation}
justifying the same notation.
\end{Construction}

\subsection{A basic case: Two lattices}
\label{sub:twola}
In this section we use results from \cite{MR3031570} to prove Theorem \ref{thm:equ} (2) if $\Gamma$ consists of two lattices. That is, we show $\widetilde{\mathcal{M}}^r(\Gamma)=\mathcal{M}(\Gamma)_k$, demonstrating some of the basic ideas for the proof of Theorem \ref{thm:equ} (2) in subsection \ref{sec:proof1}. By Lemma \ref{lem:con}, the variety $\widetilde{\mathcal{M}}^r(\Gamma)$ is contained in $\mathcal{M}(\Gamma)_k$. Since both schemes are reduced, we can deduce equality if their bivariate Hilbert polynomials coincide.\\
Each pair of two lattices is contained in a common apartment. Moreover, their convex hull forms a path between the two lattices we started with. Using the maps over the special fiber as in Construction \ref{con:var}, we obtain a complex as follows:
\[B_1\rightleftarrows\dots\rightleftarrows B_n,\]
where each $B_i=k^d$ (and each $B_i$ corresponds to a polyhedral vertex). We use the notation of the $B_i$s in order to keep track of the position.
This is the same situation as in \cite{MR3031570}:

\begin{Definition}[\cite{MR3031570}]
Let $f_i:\mathbb{P}(B_i)\dashrightarrow \mathbb{P}(B_1)\times \mathbb{P}(B_n)$ be the induced map obtained by composing the maps above along the shortest path to the extremal vertices. We define $\bigcup_{i=1}^n\overline{\mathrm{Im}(f_i)}$ to be the \textit{associated Esteves-Osserman variety}.
\end{Definition}

\begin{Remark}
A similar discussion concerning the connection between more general Esteves-Osserman varieties and linked Grassmannian can be found in \cite{hwang}. However, since this work is still in preparation, we decided to include the arguments needed for our case.
\end{Remark}

In \cite{MR3031570}, it is proved that $\widetilde{\mathcal{M}}(\Gamma)$ has the same bivariate Hilbert polynomial as $\mathbb{P}^{d-1}$ if the maps in the complex fulfill the following exactness condition:
 \setcounter{equation}{0}
\begin{align}
 \mathrm{ker}(B_i\rightarrow B_{i+1})=\mathrm{Im}(B_{i+1}\rightarrow B_i)\\
 \mathrm{ker}(B_{i+1}\rightarrow B_{i})=\mathrm{Im}(B_{i}\rightarrow B_{i+1})\\
 \mathrm{Im}(B_{i-1}\rightarrow B_{i})\cap\mathrm{ker}(B_i\rightarrow B_{i+1})=\langle0\rangle\\
 \mathrm{Im}(B_{i+1}\rightarrow B_{i})\cap\mathrm{ker}(B_i\rightarrow B_{i-1})=\langle0\rangle
\end{align}

\begin{Lemma}
\label{lem:EO}
 Let $\Gamma=\{L_1,L_2\}$ and apply Construction \ref{con:var} to obtain a complex as follows:
 \[B_1\rightleftarrows\dots\rightleftarrows B_n,\]
 where $B_i=k^d$ for all $i\in\{1,\dots,n\}$. Then this complex fulfils the conditions (1)-(4) above.
\end{Lemma}

\begin{proof}
 The two vertices lie in a common apartment. We see immediately that over $R$  the maps  \[A_i\rightarrow A_{i+1}\] are given by diagonal matrices \[A=\mathrm{diag}(a_1,\dots,a_d),\] where $a_j=1$ for $j\in J_i$ and $a_j=\pi$ for $j\in J_i^c$, where $J_i$ is some subset of $[d]$. Moreover, over $R$ the maps \[A_{i+1}\rightarrow A_i\] are given by diagonal matrices \[B=\mathrm{diag}(b_1,\dots,b_d),\] where $b_j=\pi$ for $j\in J_i$ and $b_j=1$ for $j\in J_i^c$. Thus, the first two conditions follow immediately after modding out $\pi$.\\
 The third and fourth condition follow from the fact that $J_i\subset J_{i+1}$ which is a consequence of the structure of the tropical convex hull (see Lemma \ref{prop:tropline}).
\end{proof}

Thus, we obtain Theorem \ref{thm:equ} restricted to the case where $n=2$.

\subsection{Proof of Theorem \ref{thm:equ} (1)}
\label{sub:proof2}
This subsection is devoted to proving Theorem \ref{thm:equ} (1). The strategy of the proof is to show that \[\mathcal{M}(\Gamma)_k=\widetilde{\mathcal{M}}(\Gamma)\] for convex point configurations $\Gamma$ and recover the general case by using Lemma 2.4 of \cite{cartwright2011mustafin}.\vspace{\baselineskip}

Let $\Gamma$ be a convex point configuration. We need to following Corollary, which was proved in \cite{cartwright2011mustafin}, which is essentially a consequence of Proposition \ref{prop:must} (5).

\begin{Corollary}
\label{cor:bir}
Let $\Gamma=\{L_1,\dots,L_n\}$ be an arbitrary point configuration. For every $i\in[n]$, the special fiber $\mathcal{M}(\Gamma)_k$ has a unique irreducible component which maps birationally onto $\mathbb{P}(L_i)_k$ under the projection $\mathbb{P}(L_1)_k\times\dots\times\mathbb{P}(L_n)_k\to\mathbb{P}(L_i)_k$.
\end{Corollary}

Moreover, by \cite{faltings2001toroidal}, the special fiber consists of $n$ irreducible components for convex point configurations $\Gamma$. Thus, in order to prove \[\mathcal{M}(\Gamma)_k=\widetilde{\mathcal{M}}(\Gamma)\]we need to find the unique component $C_i$ for each $i\in[n]$ with the properties as in Corollary \ref{cor:bir} in $\widetilde{\mathcal{M}}(\Gamma)$. Since there are only $n$ components those are all components and the equality is established.\vspace{\baselineskip}

Let $[L]\in\Gamma$ and consider the map $\widetilde{f_{\Gamma,[L]}}$ in Construction \ref{con:var} \[\widetilde{f_{\Gamma,[L]}}:\mathbb{P}_k^{d-1}\to\prod_{L'\in\Gamma}\mathbb{P}_k^{d-1}.\]
By construction, the map to the factor corresponding to $[L]$ in $\prod_{L'\in\Gamma}\mathbb{P}_k^{d-1}$ is the identity map. Thus, by Lemma \ref{lem:con} the closure of the image will be a component of $\mathcal{M}(\Gamma)_k$. Moreover, since the map to the factor corresponding to $[L]$ is the identity map, the respective component maps birationally to the factor $\mathbb{P}(L)_k$ under the natural projection. This completes the proof of Theorem \ref{thm:equ} for convex point configurations.\vspace{\baselineskip}

Now, let $\Gamma$ be an arbitrary point configuration. Then by Lemma 2.4 of \cite{cartwright2011mustafin}, we can compare the special fibers of $\mathcal{M}(\Gamma)$ and $\mathcal{M}(\mathrm{conv}(\Gamma))$: Let $C$ be an irreducible component of $\mathcal{M}(\Gamma)_k$, then there exists a unique irreducible component $\tilde{C}$ of $\mathcal{M}(\mathrm{conv}(\Gamma))_k$ which projects onto $C$ under \[p_{\Gamma}:\prod_{L\in\mathrm{conv}(\Gamma)}\mathbb{P}(L)_k\to\prod_{L\in\Gamma}\mathbb{P}(L)_k.\]
The key idea in proving Theorem \ref{thm:equ} (1) for arbitrary point configurations is to observe that Construction \ref{con:var} commutes with projections, i.e. for $\Gamma'=\mathrm{conv}(\Gamma)$ the following diagramm commutes:

\[\begin{tikzcd}
\mathbb{P}_k^{d-1} \arrow[rr,"\widetilde{f_{\Gamma',[L]}}"] \arrow[drr, "\widetilde{f_{\Gamma,[L]}}"]& & \prod_{L\in\Gamma'}\mathbb{P}_k^{d-1} \arrow[d,"p_{\Gamma}"]\\
 & & \prod_{L\in\Gamma}\mathbb{P}_k^{d-1}
\end{tikzcd}\]
We have proved that \[\mathcal{M}(\Gamma')_k=\bigcup_{L\in\Gamma'}\overline{\mathrm{Im}(\widetilde{f_{\Gamma',L})}}.\] To see that \[\mathcal{M}(\Gamma)_k=\bigcup_{L\in\Gamma'}\overline{\mathrm{Im}(\widetilde{f_{\Gamma,L})}}.\] we fix an irreducible component $C\in\mathcal{M}(\Gamma)_k$ and denote the component which projects onto it by $\widetilde{C}$. By the previous discussion, there exists a unique lattice $L_C\in\Gamma'$, such that $\widetilde{C}=\overline{\mathrm{Im}(\widetilde{f_{\Gamma',[L_C]}})}$. By the commutative diagram, we see that $C=p_{\Gamma}(\widetilde{C})=\overline{\mathrm{Im}(f_{\Gamma,[L_C]})}$. Thus every component $C$ is contained in $\widetilde{\mathcal{M}}(\Gamma)$ and we obtain \[\mathcal{M}(\Gamma)_k=\bigcup_{C\in\mathcal{C}_{\Gamma}} C\subset\widetilde{\mathcal{M}}(\Gamma)\subset\mathcal{M}(\Gamma)_k,\]where $\mathcal{C}_{\Gamma}$ is the set of irreducible components of $\mathcal{M}(\Gamma)_k$. We deduce \[\mathcal{M}(\Gamma)_k=\widetilde{\mathcal{M}}(\Gamma)\] for arbitrary point configurations $\Gamma$, which completes the proof.\qed

\subsection{Proof of Theorem \ref{thm:equ} (2)}
\label{sec:proof1}
In this subsection, we prove the equality of the variety $\widetilde{\mathcal{M}}^r(\Gamma)$ defined in Construction \ref{con:var} and the special fiber of the Mustafin variety $\mathcal{M}(\Gamma)$ defined in Definition \ref{def:musta}. We have already seen that $\widetilde{\mathcal{M}}^r(\Gamma)\subset\mathcal{M}(\Gamma)_k$ in Lemma \ref{lem:con}. To see the other direction, we compare the multidegrees. Since $\mathcal{M}(\Gamma)_k$ is a flat degeneration of the diagonal of $(\mathbb{P}^{d-1})^n$, we know (see e.g. \cite{cartwright2010hilbert}) that
\begin{equation}
\label{equ:mult}\mathrm{multDeg}(\mathcal{M}(\Gamma)_k)=\left\{(m_1,\dots,m_n)\in\mathbb{Z}_{\ge0}^n:\;\sum_{i=1}^nm_i=d-1\right\}.
\end{equation}
Moreover, the multidegree function takes value $1$ at each multidegree. It is easy to see that $\mathrm{multDeg}(\mathcal{M}(\Gamma)_k)={{n+d-2}\choose{d-1}}$\\
If we prove that the same holds true for the multidegree function of $\widetilde{\mathcal{M}}^r(\Gamma)$, we can deduce equality. Thus, we have to prove two statements:
\begin{itemize}
\item The special fiber of the Mustafin variety and the variety we constructed have the same multidegrees: $\mathrm{multDeg}(\widetilde{\mathcal{M}}^r(\Gamma))=\mathrm{multDeg}(\mathcal{M}(\Gamma)_k)$, where we know the right hand side of the equation.
\item The multidegree function takes value one at each $(m_1,\dots,m_n)\in\mathrm{multDeg}(\widetilde{\mathcal{M}}^r(\Gamma))$.
\end{itemize}
The basic idea is as follows: For every tuple $(m_1,\dots,m_n)$ as above, we find a vertex in $\mathrm{tconv}(\Gamma)$ whose associated variety contributes multidegree $(m_1,\dots,m_n)$. We use basic tropical intersection theory. By our method, it is immediate that there is only one vertex that can contribute a variety of this multidegree.\vspace{\baselineskip}

We begin by treating the case of lattices being in tropical general position. The case of arbitrary point configurations will follow using stable intersection.

\subsubsection{Point configurations in tropical general position}
For a point configuration $\Gamma$ in tropical general position, the number of polyhedral vertices in the respective tropical convex hull is ${{n+d-2}\choose{d-1}}$ as proved in \cite{develin2004tropical}. This coincides with the cardinality of $\mathrm{multDeg}(\mathcal{M}(\Gamma)_k)$. Therefore, the natural candidates to for the correct multidegrees are those vertices.

We begin by describing the lattices corresponding to the lattice points of the polyhedral complex associated to $\mathrm{tconv}(\Gamma)$ and start by defining a map
\begin{align}
\label{equ:C}
\begin{split}
C:V\left(\mathrm{tconv}(\Gamma))\cap\faktor{\mathbb{Z}^{d-1}}{\mathbb{Z}\textbf{1}}\right)&\longrightarrow\left\{(m_1,\dots,m_n)\in\mathbb{Z}_{\ge0}^n|\sum_{i=1}^nm_i=d-1\right\}\\v&\longmapsto (C(v)_1,\dots,C(v)_n),
\end{split}
\end{align}
where $C(v)_i$ is the largest $k$, such that $v$ lies in the codimension $k$-skeleton of the hyperplane rooted at at the lattice $L_i\in\Gamma$. It is not obvious why this map is well defined. This is where tropical intersection theory comes into play.\vspace{\baselineskip}

As mentioned before the number of vertices of $\mathrm{tconv}(\Gamma)$ coincides with the number of tuples $(m_1,\dots,m_n)\in\mathbb{Z}_{\ge0}$. If we construct one pre-image for each such tuple under the map $C$ in Equation (\ref{equ:C}), we are done. Let $(m_1,\dots,m_n)$ be such a tuple. We identify a vertex in $\mathrm{tconv}(\Gamma)$ as follows: Let $\Gamma=\{L_1,\dots,L_n\}$ and $v_1,\dots,v_n$ the corresponding vertices in the tropical torus. We intersect the standard max-hyperplanes $H_i$ at $v_i$ \[H_1^{m_1}\cap_{st}\dots\cap_{st}H_n^{m_n}=H_1^{m_1}\cap\dots\cap H_n^{m_n}=\{v\}\]as in Proposition \ref{prop:tropint}. This intersection point $v$ is vertex of $\mathrm{tconv}(\Gamma)$ and thus $C$ is well-defined. In fact we have proved that the map $C$ is bijective.

%

Now, we study the maps $\widetilde{f_{\Gamma,[L]}}=(\widetilde{g_1^{-1}},\dots,\widetilde{g_n^{-1}})$ as in Construction \ref{con:var}, in particular the associated numerical data \[d_I=\mathrm{dim}\bigcap_{i\in I}\mathrm{ker}\widetilde{g_i^{-1}}.\]
\begin{Lemma}
\label{lem:ker}
Let $v\in\mathrm{tconv}(v_1,\dots,v_n)$ such that $v$ is contained in the codimension $m_i$ skeleton at $v_i$, but not in the codimension $m_i+1$ skeleton, in the tropical torus. Let $L$ be a lattice in the class corresponding to $v$, then \[\mathrm{dim}\mathrm{Ker}(\widetilde{g_i^{-1}})=d-1-m_i,\]
where $\widetilde{g_1^{-1}}$ is the map to the $i-$th factor in Construction \ref{con:var}.
\end{Lemma}

\begin{proof}
We look at $v_i-v=(v_{i0}-v_1,\dots,v_{id}-v_d)$. Since $v$ is contained in the codimension $m_i$ skeleton at $v_i$, but not in the codimension $m_i+1$ skeleton, the minimum of $v_{i0}-v_1,\dots,v_{id}-v_d$ is attained $m_i+1$ times. As described in Construction \ref{con:var}, the map $\widetilde{g_1^{-1}}$ over the special fiber is given by
\begin{align*}D=\begin{pmatrix}a_1 & & \\
& \ddots &\\
& & a_d\end{pmatrix},
\end{align*}
where $a_j=1,$ if the minimum is attained at $v_{ij}-v_1$ and $a_j=0$ else. Thus, the dimension of the kernel of the map is given by $d-(m_i+1)$.
\end{proof}

\begin{Lemma}
\label{lem:int}
Let $v$ be a vertex of $\mathrm{tconv}(v_1,\dots,v_n)$, where the $v_i$ are in tropical general position. Let $[L]$ be the lattice class corresponding to $v$. For any $I\subset\{1,\dots,n\}$, we have \[d_I=\bigcap_{i\in I}\mathrm{ker}(\widetilde{g_1^{-1}})_i=d-1-\sum_{i\in I}m_i,\]
 where $\widetilde{g_1^{-1}}$ is the map to the $i-$th factor in Construction \ref{con:var}.
\end{Lemma}

\begin{proof}
Let $v$ be the intersection point of the max-hyperplanes at $v_1,\dots,v_n$ as before. Moreover, let $C_i$ be the codimension $m_i$ cone of the hyperplane at $v_i$ in which $v$ is contained.\\
Let $H_I$ be the intersection of the $C_i$ for $i\in I$. We can describe each $C_i$ as follows:
\begin{align*}
C_i=&\{w\in \faktor{\mathbb{R}^n}{\mathbb{R}\textbf{1}}:\mathrm{\ The\ minimum\ of\ }v_{i1}-w_1,\dots,v_{id}-w_d\\&\mathrm{\ is\ attained\ exactly\ at\ }w_j-v_{ij}\mathrm{\ for\ }j\in J_i\},
\end{align*} for some $J_i\subset\{1,\dots,d\}$ such that $|J_i|=m_i+1$. 
Since $v_1,\dots,v_n$ are in general position, $(v_i)_{i\in I}$ are in general position as well. Moreover, with the same arguments as before, the codimension $m_i$ skeleta at $v_i$ for $i\in I$ intersect in codimension $\sum_{i\in I}m_i$. Thus, we have \[\left|\bigcup_{i\in I}J_i\right|=\sum_{i\in I}m_i+1.\]
Now the kernel of the map $\widetilde{f_{\Gamma,[L]}}_i$ is generated by the basis vectors $e_j$, such that $j\notin J_i$. Thus, the intersection of the kernels is generated by those basis vectors $e_j$ such that $j\notin J_i$ for all $i$, that is \[d_I=\left|\left(\bigcup_{i\in I}J_i\right)^c\right|=d-(\sum_{i\in I}m_i+1).\] Thus the Lemma follows.
\end{proof}

\begin{Example}
Let $v_1$, $v_2$ and $v_3$ as in Example \ref{ex:tc}. We want to find the vertex which contributes the multidegree $(1,0,1)$. In order to do this, we intersect the codimension $1$ skeleton at $v_1$, the codimension $0$ skeleton at $v_2$ and the codimmension $1$ skeleton at $v_3$ as illustrated in Figure \ref{fig:proofex}. We obtain the vertex $(0,-1,-4)$, which by Example \ref{ex:must} contributes the variety $\mathbb{P}^1\times pt\times\mathbb{P}^1$, which has multidegree $(1,0,1)$.
\end{Example}

\begin{figure}
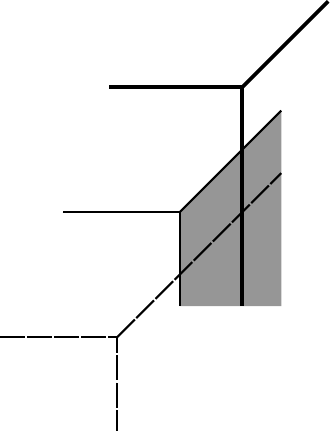
\caption{}
\label{fig:proofex}
\end{figure}

We are now ready to prove Theorem \ref{thm:equ} for the case that $\Gamma$ is in tropical general position.

\begin{proof}[Proof of Theorem \ref{thm:equ} (2) for $\Gamma$ in tropical general position]
Fix a multidegree $(m_1,\dots,m_n)\in\mathrm{multDeg}(\mathcal{M}(\Gamma)_k)$ (see Equation (\ref{equ:mult})). We pick the vertex $v$ with $C(v)=(m_1,\dots,m_n)$ (where the map $C$ is defined in Equation \ref{equ:C}) and let $[L]$ be the lattice class corresponding to $v$. We want to prove that $\mathrm{multDeg}(\overline{\widetilde{\mathrm{Im}(f_{\Gamma,[L]})}})=\{(m_1,\dots,m_n)\}$. By Theorem \ref{thm:rat}, we need to show that the tuple $(m_1,\dots,m_n)$ is the only one with $\sum m_i=d-1$ that fulfils the inequalities \[d-\sum_{i\in I} m_i>d_I\] for all $I\subset\{1,\dots,n\}$ and that there is no tuple $(n_1,\dots,n_n)$ with $\sum n_i>d-1$ that fulfils these inequalities. The first claim follows immediately from Lemma \ref{lem:ker} and \ref{lem:int}. For the second claim note that if such a tuple existed, the dimension of $\overline{\mathrm{Im}\widetilde{f_{\Gamma,[L]}}}$ would be bigger than $d-1$ which is contradiction to Lemma \ref{lem:con}. This proves that \[\mathrm{multDeg}(\widetilde{\mathcal{M}}^r(\Gamma))=\mathrm{multDeg}(\mathcal{M}(\Gamma)_k).\]
To see that the multidegree function takes value $1$ at each element in $\mathrm{multDeg}(\widetilde{\mathcal{M}}^r(\Gamma))$ we use Theorem \ref{thm:rat}: For every variety $X=X(V,V_1,\dots,V_n)$ the multidegree function takes value $1$ at each element of $\mathrm{multDeg}(X)$. Thus we need to prove that for each multidegree, there is exactly one variety contributing it. However, this is true by the arguments we needed for the set-theoric equality of the multidegrees.
\end{proof}

Note, that our method translate to the case when $\Gamma\subset\faktor{\mathbb{R}^n}{\mathbb{R}\textbf{1}}$, such that we obtain the following corollary.

\begin{Corollary}
\label{cor:gen}
Let $\Gamma\subset\faktor{\mathbb{R}^n}{\mathbb{R}\textbf{1}}$ be a point configuration in tropical general position. Then, we obtain for each $v\in V(\mathrm{tconv}(\Gamma))$
\begin{equation}
\mathrm{multdeg}(\overline{\mathrm{Im}(\widetilde{f_{\Gamma,v}})}=C(v).
\end{equation}

In particular, we obtain

\begin{equation}
\mathrm{multDeg}(\mathcal{M}^r(\Gamma))=\left\{(m_1,\dots,m_n)\in\mathbb{Z}_{\ge0}^n:\;\sum_{i=1}^nm_i=d-1\right\}
\end{equation}.
\end{Corollary}

\subsubsection{Point configurations in arbitrary position}

We now use the theory of stable intersection to deduce Theorem \ref{thm:equ} (2) for arbitrary point configurations.

\begin{proof}[Proof of Theorem \ref{thm:equ} (2) for $\Gamma$ in arbitrary position]
Let $\Gamma=\{v_1,\dots,v_n\}$ be a point configuration in arbitrary position. We slightly pertubate $v_i$ fo $\tilde{v_i}$, such that we obtain a new point configuration $\widetilde{\Gamma}$ in general position. By considering the limit $\widetilde{\Gamma}\to\Gamma$ given by $\tilde{v_i}\to v_i$, we obtain a contraction $\mathrm{tconv}(\widetilde{\Gamma})\to\mathrm{tconv}(\Gamma)$, such that any vertex $v\in \mathrm{tconv}(\Gamma)$ is the limit of several vertices $w_1,\dots,w_l$ in $\mathrm{tconv}(\widetilde{\Gamma})$. We claim that

\begin{align}\label{equ:multstable}\mathrm{multDeg}(\widetilde{f_{\Gamma,v}})=\{C(w_1),\dots,C(w_l)\},
\end{align}

then we have by Corollary \ref{cor:gen}

\begin{align*}
\mathrm{multDeg}(\widetilde{\mathcal{M}}^r(\Gamma))=\mathrm{multDeg}(\widetilde{\mathcal{M}}^r(\widetilde{\Gamma}))&=\left\{(m_1,\dots,m_n)\in\mathbb{Z}_{\ge0}^n:\;\sum_{i=1}^nm_i=d-1\right\}\\
&=\mathrm{multDeg}(\mathcal{M}({\Gamma})_k))
\end{align*}
and together with the same argument about the multidegree function as in the tropical general position case, the theorem follows. Note that 

\begin{align}
\label{equ:union}
\mathrm{multDeg}(\widetilde{\mathcal{M}}^r(\Gamma))=\bigcup_{v\in\mathrm{tconv}(\Gamma)}\mathrm{multDeg}(\widetilde{f_{\Gamma,v}}),
\end{align}

where each $\mathrm{multDeg}(\widetilde{f_{\Gamma,v}})$ is a set of multidegrees as in Equation (\ref{equ:multstable}) and the union is disjoint.\\
Thus, fix $\Gamma$ and $\widetilde{\Gamma}$ as before. Let $v\in\mathrm{tconv}(\Gamma),w_1,\dots,w_l\in\mathrm{tconv}\widetilde{\Gamma}$, such that $w_1,\dots,w_l\to v$ under the contraction $\mathrm{tconv}(\widetilde{\Gamma})\to\mathrm{tconv}(\Gamma)$. Recall, that the tropical convex hull $\mathrm{tconv}(v,v_i)$ consists of a concentation of lines $l_1,\dots,l_k$. We observe that the map $\left(\widetilde{f_{\Gamma,v}}\right)_i$ to the $i-$th factor only depends on the slopes of $l_1,\dots,l_k$. We call the collection of these slopes the \textit{combinatorial type} of $\mathrm{tconv}(v,w_i)$.\par
Under the contraction $\mathrm{tconv}(\widetilde{\Gamma})\to\mathrm{tconv}(\Gamma)$, the tropical convex hulls $\mathrm{tconv}(w_i,\tilde{v_j})$ collapses to $\mathrm{tconv}(v,v_j)$, i.e. the contraction shrinks edges of $\mathrm{tconv}(w_i,\tilde{v_j})$. In particular $\mathrm{tconv}(v,v_j)$ consists at most of as many lines as $\mathrm{tconv}(w_i,\tilde{v_j})$ and the combinatorial type of $\mathrm{tconv}(v,v_j)$ is a subset of the combinatorial type of $\mathrm{tconv}(w_j,\tilde{v_i})$. Thus, we obtain
\[\mathrm{Ker}\left(\widetilde{f_{\Gamma,v}}\right)_i\subset\mathrm{Ker}\left(\widetilde{f_{\widetilde{\Gamma},w_j}}\right)_i\]
for all $i$ and $j$. Also, as $\mathrm{tconv}(w_j,\tilde{v_i})$ collapses to $\mathrm{tconv}(v,v_i)$, there exists a vertex $\tilde{w}\in\mathrm{tconv}(w_j,\tilde{v_i})$, such that the combinatorial type of $\mathrm{tconv}(\tilde{w},\tilde{v_i})$ coincides with the combinatorial type of $\mathrm{tconv}(v,v_i)$. This vertex $\tilde{w}$ is a vertex of $\mathrm{tconv}(\widetilde{\Gamma})$ as well and is contracted to $v$. Thus, there exists $j_i$, such that $\tilde{w}=w_{j_i}$. This yields \[\mathrm{Ker}\left(\widetilde{f_{\Gamma,v}}\right)_i=\mathrm{Ker}\left(\widetilde{f_{\widetilde{\Gamma},w_{j_i}}}\right)_i.\]
We denote by $d_I^v$ and $d_I^{w_i}$ the dimensions of the intersections of the kernels with respect to $\widetilde{f_{\Gamma,v}}$ and $f_{\widetilde{\Gamma},w_i}$ respectively. We see that $d_I^v\le d_I^{w_i}$. Now, let $C(w_i)=(m_1^i,\dots,m_n^i)$. This tuple satisfies the equations $d-\sum_{j\in I}m_j^i>d_I^{v}$. We still have to prove that $M(h)=\emptyset$ for $h>|C(w_i)|$: If $M(h)\neq\emptyset$ for some $h>|C(w_i)|$, then by Theorem \ref{thm:rat} $\mathrm{dim}(\overline{\mathrm{Im}(f_L)})>d$. However, this is a contradiction to $\overline{\mathrm{Im}(\widetilde{f_{\Gamma,[L]}})}$ being contained in the special fiber of the Mustafin variety. Thus we obtain set-theoric equality of the multidegrees. As mentioned before, the multidegree functions coincide follows by the same arguments as in the tropical general position case and the fact that the union in Equation (\ref{equ:union}) is disjoint.
\end{proof}

\begin{Example}
Let $v_1=(0,0,0)$ and $v_2=(0,1,1)$. Let $\Gamma=\{L_1,L_2\}$ be the point configuration with $L_i$ being the lattice corresponding to $v_i$. We compute \[I(\mathcal{M}(\Gamma)_k)=\langle x_{22},x_{32}\rangle\cap\langle x_{11},x_{21}x_{32}-x_{31}x_{22}\rangle,\]
where $x_{1i},x_{2i},x_{3i}$ are the coordinates in the $i-$th factor. We see that $v_1$ contributes the variety corresponding to $\langle x_{22},x_{32}\rangle$ of multidegree $\{(2,0)\}$ and that $v_2$ contributes the variety corresponding to $\langle x_{11},x_{21}x_{32}-x_{31}x_{22}\rangle$ of multidegree $\{(1,1),(0,2)\}$. We want to see these multidegrees in the combinatorics of the tropical convex hull. In order to do this, we pertubate the vertex $v_2$ as illustrated in Figure \ref{fig:pert}, to obtain the point configuration $\Gamma'$ corresponding to $v_1'=(0,0,0)$, $v_2'=(0,1,2)$. We see that $v_1$ is the limit of $v_1'$ and $v_2$ is the limit of $v_2'$ and $(0,1,1)$ by reversing the pertubation. We note that $v_1'$ contributes a variety to $\mathcal{M}(\Gamma')_k$ of multidegree $\{(2,0)\}$, $(0,1,1)$ contributes a variety of multidegree $\{(1,1)\}$ and $v_2'$ contributes a variety of multidegree $\{(0,2)\}$. Then we see as in the proof that the multidegree of a vertex $v$ in $\mathrm{conv}(\Gamma)$ is given by the union of the multidegrees of the vertices whose limit is $v$ by reversing the pertubation.
\end{Example}

\begin{figure}
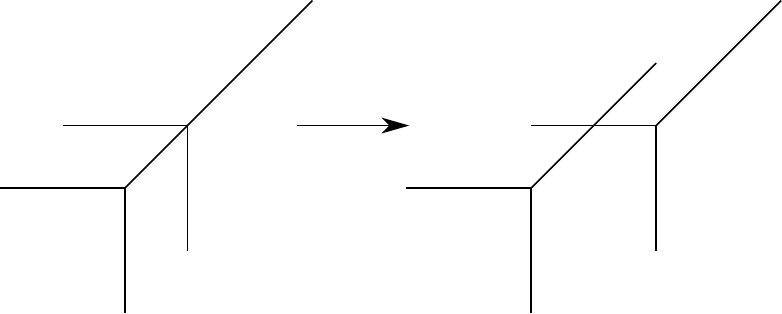
\caption{Pertubating the vertex $v_2$.}
\label{fig:pert}
\end{figure}

\subsection{Classification of the irreducible components of special fibers of Mustafin varieties}
\label{sub:class}
We begin this subsection by proving Theorem \ref{thm:class}.

\begin{proof}[Proof of Theorem \ref{thm:class}]The first part of the Theorem follows from Theorem \ref{thm:equ} and the map in Equation (\ref{equ:rat}). To see the second part, fix a variety $X(k^d;W_1,\dots,W_n)$, where $\bigcap_{i=1}^nW_i=\langle0\rangle$ and fix linear maps $g_i:k^d\to k^d$, such that $\mathrm{ker}(g_i)=W_i$. Fix a reference lattice $L$ and choose invertible $d\times d$ matrices $h_i$ over $K$, such that we obtain $g_i$ from $h_i$ by setting $t=0$. Finally, let $\Gamma=\{h_1^{-1}L,\dots,h_n^{-1}L\}$. By Lemma \ref{lem:con} $\overline{\mathrm{Im}(\widetilde{f_{[L],\Gamma}})}$ is contained in $\mathcal{M}(\Gamma)_k$. Moreover, since $\bigcap_{i=1}^nW_i=\langle0\rangle$ we obtain \[\mathrm{dim}(\overline{\mathrm{Im}(\widetilde{f_{[L],\Gamma}})})=d-1\] and thus by Lemma \ref{lem:con} $X(k^d,W_1,\dots,W_n)$ is an irreducible component as desired.
\end{proof}

Another question raised in \cite{cartwright2011mustafin} is how many irreducible components the special fiber of a Mustafin variety has. In the one apartment case, this count is inherent in the combinatorial data inherent in the tropical convex hull, which was observed in Theorem 4.4 in \cite{cartwright2011mustafin} (and also follows from Theorem \ref{thm:equ} (2)). The next remark comments on the number of irreducible components for arbitrary point configurations.

\begin{Remark}
We note that Theorem \ref{thm:equ} yields a linear algebra algorithm for counting the number of irreducible components of $\mathcal{M}(\Gamma)_k$ for arbitrary point configuration. Namely, computing the convex hull $\mathrm{conv}(\Gamma)$ and computing the number of lattice classes $[L]\in\mathrm{conv}(\Gamma)$, such that $\mathrm{dim}(\mathrm{Im}\left(\widetilde{f_{\Gamma,[L]}}\right)=d-1$, which only depends on the numerical data given by the $d_I$. This number coincides with the number of irreducible components. In the one apartment case, we can express the number of irreducible components as combinatorial data inherent in the tropical convex hull. We believe that a similar description is possible for arbitrary point configurations using the tropical description of convex hulls in a Bruhat-Tits building given in \cite{zbMATH05225139}. Computing the tropical convex hull in the Bergmann fan in the associated matroid, we can perform similar intersection products in the respective tropical linear space. More precisely, the tropical linear space will be of dimension $d-1$. Thus, for a fixed multi-degree $(m_1,\dots,m_n)$, intersecting the linear space with the codimension $m_i$ skeleta at the $i-$th vertices, we will obtain exactly one polyhedral vertex in the associated tropical convex hull. However, proving that the variety associated to the polyhedral vertex contributes in fact an irreducible component with multi-degree $(m_1,\dots,m_n)$ is not as simple as in the one apartment case. More precisely, the first step in computing the dimension of the intersections of the corresponding kernels requires knowing the basis vectors of each lattice, which in general is not as evident from the combinatorial structure as in the one apartment case.
\end{Remark}

\subsection{Mustafin varieties and Computer Vision}
Let $n>2$ and we fix $A_1,\dots,A_n$ be $3\times4$ matrices over $R$. This yields a map

\begin{align*}
f:\mathbb{P}^3_k\xrightarrow{A_1\times\cdots\times A_n}\prod_{i=1}^n\mathbb{P}^2_k.
\end{align*}

We call $\overline{\mathrm{Im}(f)}$ the associated \textit{multiview variety}. This varieties appear in computer vision: Each matrix corresponds to a camera, as \textit{taking a picture} of a geometric object in $\mathbb{P}^3$ corresponds to a linear map from $\mathbb{P}^3$ to $\mathbb{P}^2$. For further literature, we refer to \cite{zbMATH06249021,}.\par 
We obtain the following corollary of theorem \ref{thm:class}, by chossing $n$ generic $1-$dimensional subspaces $V_1,\dots,V_n$ of $\mathbb{P}^3$ and observing that the map $f$ above factorises as

\begin{equation*}
\mathbb{P}^r_k\to\prod_{i=1}^n\mathbb{P}\left(\faktor{k^4}{\mathrm{ker}(A_i)}\right)\xhookrightarrow{A_1\times\cdots\times A_n}\mathbb{P}^2_k
\end{equation*}

\begin{Corollary}
Each multiview variety appears as an irreducible component in the special fiber of a Mustafin variety.
\end{Corollary}

\section{Mustafin varieties and linked Grassmannians}
\label{sec:mulink}
Let $\Gamma\subset\mathfrak{B}_d^0$ be a convex point configuration in one apartment. We associate the following setting, which yields a prelinked Grassmannian.

\begin{enumerate}
\item The base scheme is $R$.
\item We fix a reference lattice $L=R e_1+\cdots+R e_d$.
\item We associate a graph $G=G(\Gamma)=(V(\Gamma),E(\Gamma))$ to $\Gamma$.
\item The vertex set $V(\Gamma)$ of the graph $G$ is given by the set of lattice classes in $\mathrm{conv}(\Gamma)=\Gamma$. For each vertex $v$, we denote the corresponding lattice class by $[L_v]$.
\item The edge set is given by $E(\Gamma)=\{([L_i],[L_j]):\;[L_i]\mathrm{\ is\ adjacent\ to\ }[L_j]\mathrm{\ in\ the\ building}\}$.
\item For each vertex $v$, we pick $g_v\in\mathrm{PGL}(V)$, such that $g_v L=[L_v]$.
\item Let $v,w$ be adjacent vertices. We define $g_{v,w}=g_wg_v^{-1}$. We see $g_{v,w}[L_v]=[L_w],g_{w,v}[L_w]=[L_v]$ and $g_{v,w}g_{w,v}=id$.
\item For each $v,w$, we fix a representative $f_{v,w}\in \mathrm{GL}(V)$ of $g_{v,w}^{-1}$, such that for its matrix representation $A_{v,w}=(a_{ij})_{i,j}$ with respect to $e_1,\dots,e_d$, we have $\mathrm{val}(a_{ij})\ge0$ and for at least one pair $(i',j')$, we have $\mathrm{val}(a_{i'j'})=0$.
\item We fix the vector bundle at each vertex to be $L$.
\item For an edge $e$ from $v$ to $w$, we declare the map to be $f_{v,w}:L\to L$.
\end{enumerate}

We associate a linked Grassmannian to this data.

\begin{Definition}
We denote the prelinked Grassmannian for rank $r$ subbundles of $L$ associated to the above data by $\mathrm{LG}(r,\Gamma)$.
\end{Definition}

\begin{Remark}
We note that while we have some choices in the setting above, the linked Grassmannian does not depend on those choices: As discussed below, the reference lattice $L$ and the maps $g_v$ essentially choose coordinates on the lattices $[L_v]\in\Gamma$.
\end{Remark}

\begin{Example}
We pick an apartment $A$ corresponding to the basis $e_1,\dots,e_3$. Now we take $[L_1],[L_2],[L_3]$ to be the lattice classes corresponding to $(0,0,0),(0,-1,0),(0,-2,-1)$ in the tropical torus corresponding to $A$. This is a tropically convex set, illustrated in Figure \ref{fig:linked}. The linked Grassmannian parametrises rank $1$ sub-bundles of $[L_1],[L_2],[L_3]$. Picking $L_1=Re_1+\cdots+Re_3$ as the reference lattice, the image of $[L_2]$ in $[L_1]$ is isomorphic to the sublattice of $[L_1]$ described by the image of the linear map $[L_1]\to[L_1]$ over $R$ given by the following matrix over \[\begin{pmatrix}
\pi & 0 & 0\\
0 & 1 & 0\\
0 & 0 & \pi
\end{pmatrix}.\]
We will use this description in the proof of Theorem \ref{thm:link} by choosing coordinates on $\mathrm{LG}(1,\Gamma)$ similar to the coordinates on $\mathcal{M}(\Gamma)$. 
\end{Example}

\begin{figure}
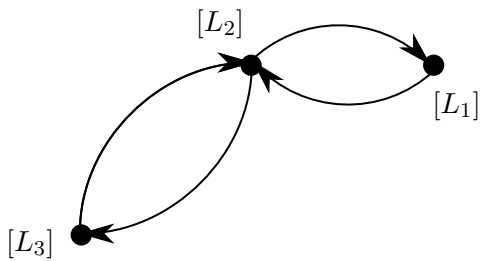
\caption{Graph associated to $(0,0,0),(0,-1,0),(0,-2,-1)$.}
\label{fig:linked}
\end{figure}

We begin by relating Falting's functor for Mustafin varieties (see theorem \ref{thm:falt}) to the linked Grassmannian problem: We fix a convex point configuration $\Gamma$ and a reference lattice $L=Re_1+\cdots+Re_d$. We consider two adjacent lattices $L',L''$, such that $L'\subset L''$. The map $g_{[L'],[L'']}$ in the setting in the beginning of this section is the inclusion map in the coordinates given by $L$.\par 
For two lattices $L',L''\in\Gamma$, we observe that both lattices are contained in a common apartment. As the inclusion operator is transitive, we consider the tropical convex hull between $[L']$ and $[L'']$ and proceed from $[L']$ to $[L'']$ along the adjacent vertices in the tropical convex hull as above, to make the same observation.\par 
Since the maps $f_{v,w}$ along the tropical convex hull in the linked Grassmannian we associated to $\Gamma$ coincide with the inclusion maps, we obtain the following result as a Corollary, interpreting Mustafin varieties as a moduli space in limit linear series theory:

\begin{thm}
\label{thm:falt}
Let $\Gamma$ be a convex point configuration, then $\mathcal{M}(\Gamma)$ represents the linked Grassmannian functor, i.e. \[\mathcal{M}(\Gamma)=\mathrm{LG}(1,\Gamma)\] as schemes.
\end{thm}

Moreover, we have thus found a class of reduced and flat linked Grassmannians.

\begin{Corollary}
Let $\Gamma$ be an arbitrary point configuration. The associated linked Grassmannians $\mathrm{LG}(1,\Gamma)$ is a flat scheme with reduced fibers.
\end{Corollary}

Most results on (pre)linked Grassmannians required the study of the simple points as elaborated in Section \ref{sec:pre}. Using Theorem \ref{thm:equ}, we can relate the simple points to Mustafin varieties. We note that it is clear that the simple points are dense in the generic fiber, since the maps between vertices over the generic fiber are isomorphisms. In the next subsection, we prove that the simple points are dense in the special fiber as well.

\subsection{Mustafin varieties and simple points}
\subsubsection{A motivating example: Two lattices}
\label{subsec:mot}
A first example for Theorem \ref{thm:falt} was given in Warning A.16 in \cite{osserman2006limit}. We pick $S=\mathrm{Spec}(k),n=d=2,r=1$. (Note that this corresponds to picking any two adjacent lattices in dimension $2$ as each two adjacent lattices will yield the same maps.) Thus the vector space associated to $v_1$ and $v_2$ is $k^2$ with maps $f_{(v_1,v_2)}=\begin{pmatrix}
1 & 0\\
0 & 0
\end{pmatrix}$ and $f_{(v_2,v_1)}=\begin{pmatrix}
0 & 0\\
0 & 1
\end{pmatrix}$. Thus the linked Grassmannian consists of points $(V_1=\left\langle\begin{pmatrix}X_0\\X_1\end{pmatrix}\right\rangle,V_2=\left\langle\begin{pmatrix}Y_0\\Y_1\end{pmatrix}\right\rangle)$, such that $f_1(V_1)\subset V_2$ and vice versa. The condition of being linked is given by $X_0Y_1=0$. In fact $\mathrm{LG}(1,\Gamma)_k$ is scheme-theoretically cut out by this equation in $\mathrm{P}^1\times\mathrm{P}^1$, which translates to a pair of $\mathbb{P}^1$'s attached at $X_0=Y_1=0$ (which is the only not exact point).\\
Given $\Gamma$ as mentioned above, we can compute the special fiber of the Mustafin variety, which by Theorem \ref{thm:equ} is given by $\overline{\mathrm{Im}(f_1)}\cup\overline{\mathrm{Im}(f_2)}$ for $f_{(v_1,v_2)}:\mathbb{P}^1\dashrightarrow \mathbb{P}^1\times \mathbb{P}^1$ given by \[\begin{pmatrix}
1 & 0\\
0 & 1
\end{pmatrix}\times \begin{pmatrix}
1 & 0\\
0 & 0
\end{pmatrix}\]
and $f_{(v_2,v_1)}:\mathbb{P}^1\dashrightarrow \mathbb{P}^1\times \mathbb{P}^1$  given by \[\begin{pmatrix}
0 & 0\\
0 & 1
\end{pmatrix}\times \begin{pmatrix}
1 & 0\\
0 & 1
\end{pmatrix}.\]
Thus, by Theorem \ref{thm:equ}, $\mathcal{M}(\Gamma)_k$ is a pair of $\mathbb{P}^1$s attached at $X_0=Y_1=0$ as well and we can conclude $\mathrm{LG}(1,\Gamma)_k=\mathcal{M}(\Gamma)_k$ in this case.
\subsubsection{Special fibers of Mustafin Varieties and prelinked Grassmannians}
The key step in proving Theorem \ref{thm:link} is the following proposition.
\begin{Proposition}
 \label{thm:equlinked}
 Let $\Gamma$ be a convex point configuration. Then the closure of the locus of simple points in $\mathrm{LG}(1,\Gamma)_k$ is set-theoretically given by the special fiber of the Mustafin varieties:
 \[\overline{\mathrm{Simp}(\Gamma)}=\mathcal{M}(\Gamma)_k.\]
\end{Proposition}

\begin{proof}
The statement is about the special fiber, thus we base change to $\mathrm{Spec}(k)$. Moreover, we can choose coordinates on the linked Grassmannian as for Mustafin varieties by performing a linear coordinate change \[\mathrm{LG}(1,\Gamma)\subset\mathbb{P}(L_1)\times\dots\times\mathbb{P}(L_n)\xrightarrow{g_1\times\dots\times g_n}\mathbb{P}(L)\times\dots\times\mathbb{P}(L)\] for a reference lattice $L$ and linear maps $g_i$, such that $g_iL=L_i$. We observe that the maps between adjacent lattices over $k$ coincide with the respective maps in Construction \ref{con:var}. First, we have to check that $G=G(\Gamma)$ together with the associated data at the vertices and edges actually satisfies the conditions in Definition \ref{def:prelink}. The only thing to check is the condition on the paths. Let $\widetilde{g_P}$ be the map obtained by the composition of maps corresponding to edges along the path $P$.\\
We make the following claim from which the required conditions follow immidiately:
\begin{Lemma}
Let $P$ be a path and $P'$ a minimal path between $[L]$ and $[L']$. Then we have \[g_P=g_{P'}\textit{ or }g_{P}\equiv0.\]
\end{Lemma}
\begin{proof}
In order to prove this claim, we consider the maps before base changing. As maps over $R$, we observe that there exists $n\in\mathbb{N}$, such that \[\pi^{n}g_P=g_{P'}.\] If $n=0$, we obtain $g_P=g_{P'}$ over the special fiber and $g_P\equiv0$ if $n>0$. Therefore, $G=G(\Gamma)$ satisfies the conditions in Definition \ref{def:prelink}.
\end{proof}

In order to link the simple points to the special fiber, we need the following lemma.

\begin{Lemma}
The tropical convex hull between two lattice classes $[L]$ and $[L']$ is a minimal path.
\end{Lemma}

In the rank $1$ case, simple points translate to the following situation: Let $(\mathcal{F}_v)_{v\in V(G)}$ be a simple point. Then there exists a $v'\in V(G)$, such that taking minimal paths $P_{v',v}$ from $v'$ to $v\in V(G)$ for each $v$, we obtain $\mathcal{F}_{v}=f_{P_{v',v}}(\mathcal{F}_v')$ and we say $(\mathcal{F}_v)_{v\in V(G)}$ is \textit{rooted} at $v'$. Thus we can classify the set $\mathrm{Simp}(\Gamma)_{v'}$ of all simple points rooted at $v'$ as the image of the following rational map\[\begin{tikzcd}g_{v'}:\mathbb{P}_k^{d-1} \arrow[rr,dashed, "(f_{P_{v',v}})_{v\in V(G)}"]& &\left(\mathbb{P}_k^{d-1}\right)^{\left|\Gamma\right|},\end{tikzcd}\]
since $\mathbb{P}_k^{d-1}$ parametrises one dimensional subvectorspaces of $k^d$. Moreover, we see immediately that \[\mathrm{Simp}(\Gamma)=\bigcup_{v\in V(G)}\mathrm{Simp}(\Gamma)_{v'}.\] It is easy to see, that the map $g_{v'}$ coincides with the map $\widetilde{f_{\Gamma,[L']}}$ constructed in Construction \ref{con:var}, where $[L']$ is the lattice class corresponding to $v'$. Therefore, we obtain \[\mathrm{Simp}(\Gamma)=\bigcup_{v\in V(\Gamma)}\mathrm{Simp}(\Gamma)_v=\bigcup_{[L]\in \mathrm{conv}(\Gamma)} \mathrm{Im}(\widetilde{f_{\Gamma,[L]}})\] and by taking the closure and applying Theorem \ref{thm:equ} we obtain \[\overline{\mathrm{Simp}(\Gamma)}=\mathcal{M}(\Gamma)_k\] as desired.
\end{proof}

Since $\mathcal{M}(\Gamma)=\mathrm{LG}(1,\Gamma)$ by Theorem \ref{thm:falt}, we see that $\mathrm{LG}(1,\Gamma)_k=\mathcal{M}(\Gamma)_k=\overline{\mathrm{Simp}(\Gamma)}$, which proves Theorem \ref{thm:link}.

\section{Mustafin varieties and local models of Shimura varieties}
\label{sec:shi}
We interpret the standard local model $M^{loc}$ described in subsection \ref{subsec:shim} as a Mustafin variety by the following theorem:

\begin{thm}
\label{thm:shim}
Let $\Gamma=\{L_0,\dots,L_{d-1}\}$ the configuration chosen in subsection \ref{subsec:shim}, then \[\mathcal{M}(\Gamma)= M^{loc},\] where $M^{loc}$ is the local model for $r=1$.
\end{thm}

\begin{Remark}
By a similar argument $M^{loc}$ is shown to coincide with a different class of linked Grassmannian for rank $r$ subbundles in \cite{hwang} .
\end{Remark}

\begin{proof}
We prove that $M^{loc}=\mathrm{LG}(1,\Gamma)$ and follow with Theorem \ref{thm:falt} that $M^{loc}=\mathcal{M}(\Gamma)$. Fix an $R-$scheme $S$ and let $(\mathcal{F}_0,\dots,\mathcal{F}_{d-1})$ be a tuple of rank $1$ sub-bundles of $(L_{1,S},\dots,L_{d-1,S})$ parametrised by $\mathrm{LG}(1,\Gamma)$. Since the maps \[L_0\to\cdots\to L_{d-1}\] appear in underlying graph of the linked Grassmannian associated to $\Gamma$, we see that the tuple $(\mathcal{F}_0,\dots,\mathcal{F}_{d-1})$ is parameterised by $M^{loc}$ as well. For the other direction, let $(\mathcal{F}_0,\dots,\mathcal{F}_{d-1})$ be a tuple as above parametrised by $M^{loc}$. In order to see that this tuple is parametrised by $\mathrm{LG}(1,\Gamma)$, we need to prove that for $i,j\in\{0,\dots,d-1\}$, the map $L_i\to L_j$ maps $\mathcal{F}_i$ to $\mathcal{F}_j$. However, the map $L_i\to L_j$ is given by the composition $L_i\to L_{i+1}\cdots\to L_{j-1}\to L_j$. Since the tuple $(\mathcal{F}_i)_i$ is parametrised by $M^{loc}$ the chain of maps $\mathcal{F}_i\to\mathcal{F}_{i+1}\to\cdots\mathcal{F}_{j-1}\to\mathcal{F}_j$ is well defined, $\mathcal{F}_i$ is mapped to $\mathcal{F}_j$ as well.
\end{proof}

Since Mustafin varieties are flat with reduced fibers by Proposition \ref{prop:must}, we can use Theorem \ref{thm:shim} to obtain a new proof of Theorem \ref{thm:gortz} for $r=1$.

\bibliographystyle{alpha}
\bibliography{literature.bib}

\begin{thebibliography}{{Mum}72}

\bibitem[AB08]{buldingsbook}
Peter {Abramenko} and Kenneth~S. {Brown}.
\newblock {\em {Buildings. Theory and applications.}}
\newblock Berlin: Springer, 2008.

\bibitem[ABGJ15]{zbMATH06514534}
Xavier {Allamigeon}, Pascal {Benchimol}, St\'ephane {Gaubert}, and Michael
  {Joswig}.
\newblock {Tropicalizing the simplex algorithm.}
\newblock {\em {SIAM J. Discrete Math.}}, 29(2):751--795, 2015.

\bibitem[AR10]{allermann2010first}
Lars Allermann and Johannes Rau.
\newblock First steps in tropical intersection theory.
\newblock {\em Math. Z.}, 264(3):633--670, 2010.

\bibitem[AST13]{zbMATH06249021}
Chris {Aholt}, Bernd {Sturmfels}, and Rekha {Thomas}.
\newblock {A Hilbert scheme in computer vision.}
\newblock {\em {Can. J. Math.}}, 65(5):961--988, 2013.

\bibitem[CHSW11]{cartwright2011mustafin}
Dustin Cartwright, Mathias H{\"a}bich, Bernd Sturmfels, and Annette Werner.
\newblock {Mustafin Varieties}.
\newblock {\em Selecta Math.}, 17(4):757--793, 2011.

\bibitem[CLZ16]{castillo2016representable}
Federico Castillo, Binglin Li, and Naizhen Zhang.
\newblock {Representable Chow classes of a product of projective spaces}.
\newblock {\em arXiv preprint arXiv:1612.00154}, 2016.

\bibitem[CS10]{cartwright2010hilbert}
Dustin Cartwright and Bernd Sturmfels.
\newblock {The Hilbert scheme of the diagonal in a product of projective
  spaces}.
\newblock {\em IMRN}, 2010(9):1741--1771, 2010.

\bibitem[DS04]{develin2004tropical}
Mike Develin and Bernd Sturmfels.
\newblock Tropical convexity.
\newblock {\em Doc. Math}, 9:1--27, 2004.

\bibitem[DY07]{zbMATH05246540}
Mike {Develin} and Josephine {Yu}.
\newblock {Tropical polytopes and cellular resolutions.}
\newblock {\em {Exp. Math.}}, 16(3):277--291, 2007.

\bibitem[EO13]{MR3031570}
Eduardo Esteves and Brian Osserman.
\newblock Abel maps and limit linear series.
\newblock {\em Rend. Circ. Mat. Palermo (2)}, 62(1):79--95, 2013.

\bibitem[Fal01]{faltings2001toroidal}
Gerd Faltings.
\newblock Toroidal resolutions for some matrix singularities.
\newblock In {\em Moduli of Abelian Varieties}, pages 157--184. Springer, 2001.

\bibitem[G{\"o}r01]{gortz2001flatness}
Ulrich G{\"o}rtz.
\newblock {On the flatness of models of certain Shimura varieties of PEL-type}.
\newblock {\em Math. Ann.}, 321(3):689--727, 2001.

\bibitem[HL]{hwang}
Brian Hwang and Binglin Li.
\newblock {Linked Grassmannians and local models of Shimura varieties:
  Unramified type A}.
\newblock {\em In preparation}.

\bibitem[HO08]{helm2008flatness}
David Helm and Brian Osserman.
\newblock {Flatness of the linked Grassmannian}.
\newblock {\em Proceedings of the American Mathematical Society},
  136(10):3383--3390, 2008.

\bibitem[JSY07]{zbMATH05225139}
Michael {Joswig}, Bernd {Sturmfels}, and Josephine {Yu}.
\newblock {Affine buildings and tropical convexity.}
\newblock {\em {Albanian J. Math.}}, 1(4):187--211, 2007.

\bibitem[JY16]{jensen2016stable}
Anders Jensen and Josephine Yu.
\newblock Stable intersections of tropical varieties.
\newblock {\em Journal of Algebraic Combinatorics}, 43(1):101--128, 2016.

\bibitem[Kat12]{katz2012tropical}
Eric Katz.
\newblock Tropical intersection theory from toric varieties.
\newblock {\em Collect. Math.}, 63(1):29--44, 2012.

\bibitem[KT06]{keel2006geometry}
Sean Keel and Jenia Tevelev.
\newblock {Geometry of Chow quotients of Grassmannians}.
\newblock {\em Duke Math.}, 134(2):259--311, 2006.

\bibitem[Li17]{Li13}
Binglin Li.
\newblock Images of rational maps of projective spaces.
\newblock {\em IMRN}, 2017.

\bibitem[MS15]{MR3287221}
Diane Maclagan and Bernd Sturmfels.
\newblock {\em Introduction to tropical geometry}, volume 161 of {\em Graduate
  Studies in Mathematics}.
\newblock American Mathematical Society, Providence, RI, 2015.

\bibitem[{Mum}72]{zbMATH03362020}
David {Mumford}.
\newblock {An analytic construction of degenerating curves over complete local
  rings.}
\newblock {\em {Compos. Math.}}, 24:129--174, 1972.

\bibitem[Mus78]{mustafin1978nonarchimedean}
G.A. Mustafin.
\newblock Nonarchimedean uniformization.
\newblock {\em Mathematics of the USSR-Sbornik}, 34(2):187, 1978.

\bibitem[Oss06]{osserman2006limit}
Brian Osserman.
\newblock A limit linear series moduli scheme.
\newblock In {\em Annales de l'institut Fourier}, volume~56, pages 1165--1205,
  2006.

\bibitem[Oss14]{osserman2014limit}
Brian Osserman.
\newblock Limit linear series moduli stacks in higher rank.
\newblock {\em arXiv preprint arXiv:1405.2937}, 2014.

\bibitem[PRS13]{PRSLocal}
Georgios {Pappas}, Michael {Rapoport}, and Brian {Smithling}.
\newblock {Local models of Shimura varieties, I. Geometry and combinatorics.}
\newblock In {\em {Handbook of moduli. Volume III}}, pages 135--217.
  Somerville, MA: International Press; Beijing: Higher Education Press, 2013.

\bibitem[{Rau}16]{rau2008intersections}
Johannes {Rau}.
\newblock {Intersections on tropical moduli spaces.}
\newblock {\em {Rocky Mt. J. Math.}}, 46(2):581--662, 2016.

\bibitem[Sha13]{shaw2013tropical}
Kristin~M Shaw.
\newblock A tropical intersection product in matroidal fans.
\newblock {\em SIAM Journal on Discrete Mathematics}, 27(1):459--491, 2013.

\bibitem[{Wer}11]{zbMATH05900629}
Annette {Werner}.
\newblock {A tropical view on Bruhat-Tits buildings and their
  compactifications.}
\newblock {\em {Cent. Eur. J. Math.}}, 9(2):390--402, 2011.

\end{thebibliography}
\Addresses
\end{document}